\documentclass{article}
\usepackage[utf8]{inputenc}
\usepackage{authblk}
\DeclareMathSizes{12}{30}{8}{12}

\usepackage{amsfonts}
\usepackage{amsmath}
\usepackage{framed}
\usepackage[ruled,vlined]{algorithm2e}
\usepackage{algorithmic}

\usepackage{verbatim}
\usepackage{listings}
\usepackage{fancyvrb}
\usepackage{cancel}

\allowdisplaybreaks

\usepackage[english]{babel}
\usepackage{amsthm}
\usepackage{framed}
\usepackage{tikz}
\usepackage{bbm}

\usepackage{ulem}
\usepackage{amssymb}
\usepackage{mathtools}
\usepackage{commath}
\usepackage{graphicx} 

\newtheorem{theorem}{Theorem}
\newtheorem{lemma}{Lemma}

\newtheorem{corollary}{Corollary}

\newtheorem{remark}{Remark}

\newcommand{\R}{\mathbb{R}}
\newcommand{\N}{\mathbb{N}}
\newcommand{\E}{\mathbb{E}}

\newcommand{\M}{\mathcal{M}}
\newcommand{\rk}{\text{rk}}

\newcommand{\Poi}{{\rm Poi}}
\newcommand{\Bin}{{\rm Bin}}
\newcommand{\Yone}{H^{(1)}}
\newcommand{\Ytwo}{H^{(2)}}

\newcommand{\hatCone}{\widehat{C}^{(1)}}
\newcommand{\hatLone}{\widehat{L}^{(1)}}

\newcommand{\hatCl}{\widehat{C}^{(\ell)}}
\newcommand{\hatLl}{\widehat{L}^{(\ell)}}
\newcommand{\Mk}{M^{(k)}}

\newcommand{\Pk}{P^{(k)}}

\newcommand{\hatPk}{\widehat P^{(k)}}

\newcommand{\normsquare}[1]{\big\|#1\big\|_{\square}}
\newcommand{\barc}{14}
\newcommand{\FK}{\mathcal{F}_K}
\newcommand{\Supp}{\operatorname{Supp}}

\newcommand{\ju}[1]{\textcolor{blue}{Julien: #1}}

\title{Generalized multi-view model:\\ Adaptive density estimation under low-rank constraints}
\author[1]{Julien Chhor}
\author[2,3]{Olga Klopp}
\author[3]{ Alexandre B. Tsybakov}

\affil[1]{Toulouse School of Economics, France}
\affil[2]{ESSEC Business School, France}
\affil[3]{CREST, ENSAE, Institut Polytechnique de Paris, France}

\begin{document}

\maketitle

\begin{abstract}
We study the problem of bivariate discrete or continuous probability density estimation under low-rank constraints.
For discrete distributions, we assume that the two-dimensional array to estimate is a low-rank probability matrix. 
In the continuous case, we assume that the density with respect to the Lebesgue measure satisfies a generalized multi-view model, meaning that it is $\beta$-H\"older and can be decomposed as a sum of $K$ components, each of which is a product of one-dimensional functions.
In both settings, we propose estimators that achieve, up to logarithmic factors, the minimax optimal convergence rates under such low-rank constraints.
In the discrete case, the proposed estimator is adaptive to the rank $K$. 
In the continuous case, our estimator converges with the $L_1$ rate $\min((K/n)^{\beta/(2\beta+1)}, n^{-\beta/(2\beta+2)})$ up to logarithmic factors, and it is adaptive to the unknown support as well as to the smoothness $\beta$ and to the unknown number of separable components~$K$. 
We present efficient algorithms to compute our estimators.
%

\end{abstract}

\section{Introduction}

Estimating discrete and continuous probability distributions is one of the fundamental problems in statistics and machine learning. 
A classical density estimator for both discrete and continuous data is the histogram, while for continuous densities the most popular method is kernel density estimator (KDE) see, e.g., \cite{Silverman86,DevroyeGyofri, Scott1992Multivariate,klemela2009smoothing,tsybakov2009introduction,Gramacki2017NonparametricKDE,wang2019nonparametric}.
Asymptotically, these estimators can consistently recover any probability density on $\mathbb{R}^m$ in the total variation ($L_1$) distance based on $n$ independent identically distributed (iid) observations, if $nh^{m}\rightarrow \infty$, where $h$ is the tuning parameter (bin width for the histogram in the continuous case and bandwidth for the KDE), and we assume that $h\rightarrow 0$, $n\rightarrow \infty$, see, e.g., \cite {DevroyeGyofri}. On the other hand, smoothness assumptions on the underlying density are not enough to grant good accuracy of these estimators when $m$ is large. Their rate of convergence drastically deteriorates with the dimension even if $h$ is chosen optimally, and it remains true for any estimators under only smoothness assumptions. This gives rise to suggestions of density estimators that overcome the curse of dimensionality under more assumptions on the underlying density than only smoothness. An early suggestion is Projection Pursuit density estimation (PPDE) \cite{FriedmanStuetzle1984projection-pursuit}, which is an iterative algorithm to estimate the density by finding a subspace spanned by a small number of significant components. One may consider PPDE as being adapted to the setting where there exists a linear map that transforms the underlying random vector to a smaller dimensional random vector with independent components. This technique is rather popular but, to the best of our knowledge, theoretical guarantees on the performance of PPDE are not available, see a recent survey on PPDE in \cite{wang2019nonparametric}. A related dimension reduction model for density estimation arises from independent component analysis (ICA), where one assumes the existence of a linear bijection of the underlying random vector to a random vector of the same dimension with independent components \cite{samarov2004nonparametric}. It is shown that under this model there is no curse of dimensionality in the sense that there exist estimators achieving one-dimensional rates \cite{samarov2004nonparametric,samarov2005aggregated,amato2010noisy,v2020structural}. 
Finally, a recent line of work starting from \cite{pmlr-v32-songa14} and further developed in \cite{9779133,9740538, pmlr-v89-kargas19a, 10.5555/2999792.2999972, vandermeulen2021beyond,Vandermeulen2023} deals with the multi-view model for density estimation, that is, a finite mixture model whose components are products of one-dimensional probability densities.  
Densities  $f:\mathbb{R}^m\to \mathbb{R}$ satisfying the multi-view model are the form
\begin{equation}
  \label{basic-multi-view}
f(x)=\sum_{i=1}^K w_i\prod_{j=1}^m f_{ij}(x^T e_j) \quad \text{with} \quad \sum_{i=1}^K w_i=1, w_i\ge 0,  
\end{equation}
  where $e_j$'s are the canonical basis vectors in $\mathbb{R}^m$ and $ f_{ij}$'s are one-dimensional probability densities. Weights $w_i$ and $ f_{ij}$'s are unknown.
The fact that this model is free from the curse of dimensionality when $f_{ij}$'s are Lipschitz continuous and supported of $[0,1]$ is demonstrated in  \cite{vandermeulen2021beyond} through a theoretical analysis of its sample complexity. However, \cite{vandermeulen2021beyond} does not develop computationally tractable estimators. 
\vspace{1mm}
%


 In this paper, we focus on two-dimensional density estimation and consider a new model that generalizes the multi-view model in two aspects. We call it the {\it generalized multi-view model}. First, in contrast to the usual multi-view model, we do not assume the additive components to be products of densities but rather products of arbitrary functions. Second, we do not assume these functions to be Lipschitz continuous. We only need a $\beta$-H\"older continuity for some $\beta\in(0,1]$ of the overall two-dimensional density, which is a sum of $K$ such products. We propose a new estimator that is both computationally tractable and offers improved statistical guarantees
 by achieving the one-dimensional estimation rate to within a logarithmic factor. 
 Our analysis deals with the $L_1$ risk of density estimators. 
As argued in \cite{DevroyeGyofri}, using the $L_1$ norm to characterize the error of density estimation has several advantages. Indeed, the $L_1$ distance between two densities is equivalent, up to a multiplicative constant factor, to the total variation distance between the corresponding probability measures. 
Moreover, the $L_1$ risk for density estimation is transformation invariant, which is not the case for the $L_2$ risk  most often studied in the literature.
From the technical point of view, it is more difficult to deal with the $L_1$ risk than with the $L_2$ risk. 

\vspace{1mm}

Our approach to tackling generalized multi-view models is based on a reduction to the problem of estimating multivariate discrete distributions under a low-rank structure that we also consider in detail.
This problem is of independent interest and it arises in many applications, 
in particular, if the aim is to explore correlations between categorical random variables, which is a particularly relevant subject \cite{johndrow2017tensor,dunson2009nonparametric,diakonikolas2019robust,tahmasebi2018identifiability}.
Without structural assumptions, estimating a discrete distribution on a set of cardinality $D$ in total variation distance 
is associated with the minimax risk of the order of $\sqrt{{D}/{n}}$, where $n$ is the number of observations (see \cite{kamath2015learning,han2015minimax} and Corollary~\ref{cor:minimax rate general discrete} below).
However, by imposing certain structure assumptions, it is possible to reduce the estimation risk. Several assumptions, including monotonicity, unimodality, $t$-modality, convexity, log-concavity or $t$-piecewise degree $k$-polynomial structure have been considered in the literature (see, for example, ~\cite{canonne2018testing,diakonikolas2015differentially,durot2013least}).
In the present paper, we deal with a different setting where a multivariate discrete distribution is estimated under a low-rank structure. 
Specifically, for two integers $d_1, d_2 \geq 2$, we consider the problem of estimating a discrete distribution on a set of cardinality $D=d_1d_2$ defined by a matrix of probabilities $P = (P_{ij})_{i \in [d_1] ,j \in [d_2]}$ with rank at most $K \geq 1$. 
This setting arises, for example, in the analysis of data represented as a table with $n$ rows and 2 columns.
Each row corresponds to one individual, while each column contains information about one feature of the individual, for instance, $(1)$ eye color, and  $(2)$ hair color.
We assume that the value of the cell in the table is a categorical random variable, for example, that hair color can only take $6$ possible values: black, brown, red, blond, gray, white. 
Thus, each element of the $\ell$-th column has a discrete distribution over $\{1,\dots,d_\ell\}$, for some $d_\ell \geq 2$, $\ell=1,2$.
The individuals are assumed to be iid but the columns can be correlated. For instance, hair color can be correlated with eye color.
In other words, each row can be viewed as a realization of a pair of correlated discrete random variables.
If we are interested in possible associations between the two variables, we are lead to estimating their joint distribution. Assuming that $P$ has low rank means that there exists a reduced representation of the correlation structure.
A basic unbiased estimator of $P$ is a histogram $Y/n$ where $Y$ is a $d_1\times d_2$ matrix such that its $(i,j)$-th entry $Y_{ij}$ is the number of individuals whose row in the data table equals $(i,j)$. 
However, the histogram does not take advantage of low-rank structure and it only attains the slow rate $\sqrt{d_1d_2/n}$ in the total variation distance. We will suggest an estimator attaining a faster rate and show that it is minimax optimal up to logarithmic factors.  

\section{Motivation}

In many applications, one needs to explore relations between two objects that
may have a complex structure, yet are linked via a low-dimensional latent space.
This situation can be often described by mixture models and low-rank matrix models. For the problem with discrete distributions,
one of the important examples is given by the {\it probabilistic Latent Semantic Indexing} framework 
for topic models \cite{Hofmann1999}. It assumes that co-occurrences
of words and documents are independent given one of $K$ latent topic classes. Then the
joint probability matrix of words and documents is a mixture of at most $K$ matrices and its rank 
does not exceed $K$, which is typically a small number. Another example of low-rank probability matrix estimation is provided by the Stochastic Block Model \cite{Holland1983StochasticBF,abbe/sbm}. In this case, the problem is to estimate the matrix of connection probabilities of a random graph under the assumption that its nodes fall into $K$ groups with constant connection probabilities within and between each two groups. Such a probability matrix is of rank at most $K$. Low-rank probability matrix estimation problems also arise in the context of  collaborative filtering, matrix completion and other \cite{journals/focm/CandesR09,srebro_collaborative}.

For the problems characterized by continuous probability densities,  multi-view models provide a nonparametric analog of classical mixture models. In contrast to these classical models, they do not assume that the components of the mixture depend on finite number of parameters but rather consider them as functions satisfying some general constraints, such as smoothness or just integrability. In model~\eqref{basic-multi-view}, the resulting function $f$ is the probability density of a random vector \( X = (x_1,\dots, x_m) \in [0,1]^m\) with entries $x_1,\dots, x_m$ that are independent conditional on a latent variable that can take \( K \) distinct values. 
The generalized multi-view model considered in this paper is broader than the basic multi-view model~\eqref{basic-multi-view} as it  allows the functions \( f_{ij} \) to be integrable real-valued functions rather than densities. Also, we only assume H\"older smoothness of $f$ and not of all the components \( f_{ij} \).    Nevertheless, we will demonstrate that estimation over the class of generalized multi-view models is, up to logarithmic factors, not harder than estimation over the subclass described by~\eqref{basic-multi-view}. 
In this paper, we focus on the setting, where the aim is to explore relations between two variables ($m=2$) and we explicitly construct polynomial-time estimators achieving the optimal rates for such models. 

A relevant question is to check whether the multi-view model holds for a given particular problem in practice. 
We address this issue by providing estimators that are adaptive 
to the unknown number of components $K$ varying on a wide scale of values. Very large values of $K$ correspond to the absence of low-rank structure. For such $K$, our estimator achieves the same rate as the usual nonparametric density estimator of a smooth density (with no additional structure), and we show that this is optimal. 
In other words, our adaptive estimator achieves the minimax optimal rate regardless of whether the multi-view model holds or not. 
Thus, adaptation guarantees that checking the low-rank assumption is not necessary in practice.

\section{Summary of contributions and related work}


The importance of low-rank structures in nonparametric density estimation has been discussed in several papers suggesting and analyzing various estimation methods \cite{9779133,9740538, pmlr-v89-kargas19a, pmlr-v32-songa14,  10.5555/2999792.2999972, vandermeulen2021beyond}. In \cite{pmlr-v32-songa14}, the authors introduced the multi-view model and proposed a kernel method for learning it but did not study whether it can lead to an improvement in nonparametric density estimation. 
The paper \cite{10.5555/2999792.2999972} provided empirical evidence that hierarchical low-rank decomposition of kernel embeddings can lead to improved performance in density estimation. 
More recently, \cite{9779133,9740538} used a low-rank characteristic function to improve nonparametric density estimation. Finally, \cite{vandermeulen2021beyond} proved that there exists an estimator that converges at a rate $O\big(n^{-1/3}\big)$ in the $L_1$ norm
to any density satisfying the multi-view model with Lipschitz continuous marginals. This estimator is not constructed explicitly and is not computationally tractable in general. Furthermore, \cite{vandermeulen2021beyond} shows that the standard histogram estimator
can converge at a rate slower than $O\big(n^{-1/m}\big)$ on the same class of densities, where $m$ is the dimension of the data. 

\vspace{1mm}
 
A related problem of estimating a low-rank probability matrix from discrete counts has been considered in~\cite{Jain2020linear}, where the authors propose a polynomial time algorithm (called the curated SVD) for the case of square matrix, $d_1=d_2=d$, and requiring the exact knowledge of the rank $K$. They prove an upper bound on the total variation error of the curated SVD that scales as $\psi(K,d,n):= \sqrt{\frac{Kd}{n}} \land 1$ (to within a logarithmic factor in $K,n$) with probability at least $1-d^{-2}$. They also state a lower bound with the rate $\psi(K,d,n)$ {\it in expectation} by referring to the lower bounds of ~\cite{han2015minimax,kamath2015learning} for general discrete distributions on a set of cardinality $D=Kd$. Based on that, the authors in~\cite{Jain2020linear} claim minimax optimality of the rate $\psi(K,d,n)$, up to logarithmic factors. 
However, the lower bound obtained by this argument only holds under significant restrictions on $K,d,n$ that are not specified
in~\cite{Jain2020linear}. 
Indeed, the lower bounds of~\cite{han2015minimax,kamath2015learning} for general discrete distributions are only meaningful under some specific conditions on $D,n$. For example, the lower bound of \cite[Theorem 1]{han2015minimax} is vacuous for $D\asymp 1$ and the lower bound of \cite[Lemma 8]{kamath2015learning} is vacuous for $D\asymp \sqrt{n}$. These lower bounds are established only in expectation and it is not legitimate to  compare them directly with upper bounds in probability derived in~\cite{Jain2020linear}. The upper bound in probability in~\cite[Theorem 2]{Jain2020linear} can be transformed into a bound in expectation at the expense of adding a $O(d^{-2})$ term to the rate. This imposes one more restriction $d^{-2} \lesssim \psi(K,d,n)$ to match the lower bound up to a logarithmic factor.
The algorithm suggested in~\cite{Jain2020linear} is based on normalization of the probability matrix by rescaling each row and column. 
Similar ideas have been developed in the literature on topic models, where topic matrices are estimated using an SVD-based technique on a rescaled corpus matrix~\cite{ke2022using}.
Our approach is different and based on novel techniques that we call a {\it localized SVD denoising}. 
The localized SVD algorithm that we suggest does not need to rescale the input matrix and relies on splitting the matrix into sub-matrices with entries of similar order of magnitude. 
It may be of interest in other contexts as well.

\vspace{2mm}

The contributions of the present work are as follows. 
\begin{itemize}
    \item  We prove minimax lower bounds in the total variation distance for general discrete distributions on a set of cardinality $D$. We generalize~\cite{han2015minimax,kamath2015learning} in the sense that we derive lower bounds not only in expectation but also in probability and, in contrast to those works, we obtain the lower rate $ \sqrt{\frac{D}{n}} \land 1$ for all $D,n\ge 1$ with no restriction. Next, under the low-rank matrix structure, we prove lower bounds of the order of $\psi(K,d,n)$  both in expectation and in probability, with no restriction on $K,d,n$, where $d=d_1\vee d_2$.
    Moreover, we propose a computationally efficient algorithm to estimate a low-rank probability matrix  $P$ and show that it attains the same rate $\psi(K,d,n)$ up to a logarithmic factor. Thus, we prove the minimax optimality of this rate and of our algorithm, up to a logarithmic factor. Unlike the curated SVD of~\cite{Jain2020linear}, our algorithm applies to non-square matrices and it is adaptive to the unknown rank $K$. 
\item We propose a method of estimating $\beta$-H\"older densities for $\beta\in(0,1]$ under the generalized multi-view model.  
Our algorithm achieves the rate of convergence $\left (K/n\right)^{\beta/(2\beta+1)} \wedge n^{-\beta/(2\beta+2)}$ up to a logarithmic factor on the class of densities that are (i) $\beta$-H\"older over an {\it unknown} sub-rectangle of $[0,1]^2$ and (ii)~represented as a sum of $K$ separable components. In the two-dimensional case that we consider, we improve upon the prior work~\cite{vandermeulen2021beyond} in the following aspects:
\begin{itemize}
 \item Our estimator is computationally tractable.
\item The study in \cite{vandermeulen2021beyond} was devoted to the case $\beta=1$ and the standard multi-view model while we provide an extension to the generalized multi-view model described above and to any $\beta\in(0,1]$.
   \item We prove a lower bound showing that the above convergence rate is minimax optimal up to a logarithmic factor on the class of densities satisfying the generalized multi-view model. We establish the explicit dependence of the minimax rate on $K,n,\beta$ revealing, in particular, that it exhibits an elbow at $K\asymp n^{1/(2\beta+2)}$. 
   We note that our lower bound is stronger since we prove it for the smaller class $\mathcal{G}_{K,\beta}^\circ$ of densities  satisfying the standard multi-view model \eqref{basic-multi-view} with $m=2$, where $f_{ij}$'s are probability densities on $[0,1]$ that are $\beta$-Hölder on their support.
   \item We propose an estimator that is adaptive to the unknown number of separable components $K$, to the unknown smoothness $\beta$, and to the unknown support of the density. 
   As shown by our lower bound, this estimator also reaches the minimax optimal convergence rate, up to a logarithmic factor, on the class $\mathcal{G}_{K,\beta}^\circ$. 
   It can therefore be employed for learning mixture models from the class $\mathcal{G}_{K,\beta}^\circ$ while guaranteeing robustness to the model misspecification since it attains a comparable rate over the substantially larger class of generalized multi-view models.
   
\end{itemize}
\item  We provide a package for computation avalable at  \texttt{https://github.com/hi-paris/Lowrankdensity}.
We run a numerical experiment demonstrating the efficiency of our estimators both in discrete and continuous settings.
\end{itemize}

\section{Notation}

For two real numbers $x,y$, we define $x\land y := \min(x,y)$ and $x \lor y := \max(x,y)$. 
For $d \in \N$, we set $[d] = \{1, \dots, d\}$. 
For any probability vector or probability matrix  $P$ and for any $n \in \N^*$, we denote by $\M(P,n)$ the multinomial distribution with probability parameter $P$ and sample size $n$. 

For any matrix $\Lambda$, we denote by $\Lambda_{ij}$ its $(i,j)$th entry and by $\rk(\Lambda)$ its rank. 
We denote by $\Lambda_{IJ} = (\Lambda_{ij})_{i\in I, j \in J}$ an extraction of matrix $\Lambda\in \mathbb{R}^{d_1\times d_2}$ corresponding to the sets of indices $I \subseteq [d_1]$ and  $J \subseteq [d_2]$.
We will use several matrix norms, namely, the operator norm denoted by $\|\Lambda\|$, the nuclear norm $\|\Lambda\|_*$, the Frobenius norm $\|\Lambda\|_F$, the entry-wise $\ell_1$-norm $\|\Lambda\|_{1}$, and the norms
\begin{align*}
    &\|\Lambda\|_{1,\infty} = \max_{j\in[d_2]} \sum_{i=1}^{d_1} |\Lambda_{ij}|, 
    \\
    & \normsquare{\Lambda} = \|\Lambda\|_{1,\infty} \lor \|\Lambda^\top\|_{1,\infty}.
\end{align*}
The notation $\|\cdot\|_{L_1}$ is used for the norms in $L_1([0,1], \rm{Leb})$ and in $L_1([0,1]^2, \rm{Leb})$, where $\rm{Leb}$ denotes the Lebesgue measure.

We denote by $\|\cdot\|$ the Euclidean norm in $\R^2$, by $\Supp (f) = f^{-1}\left(\R\setminus\{0\}\right)$ the support of a real-valued function $f$,  by $\mathbbm{1}_A(\cdot)$ the indicator function of set $A$, and by $|A|$ the cardinality of finite set $A$.
Throughout the paper, the absolute positive constants are denoted by $C$ and may take different values on each appearance and we assume, unless otherwise stated, that $n, d_1, d_2 \geq 2$.


\section{Discrete distributions}
\label{sec:discrete}

Consider first the setting with discrete distributions, which provides a base for studying continuous distributions in Section~\ref{sec:densities}.

For integer $K \geq 1$, let $\mathcal{T}_K$ be the class of all probability matrices of rank at most $K$: 
\begin{equation}\label{def_DK}
    \mathcal{T}_K = \bigg\{P \in \mathbb{R}^{d_1 \!\times\! d_2} ~ \Big| ~ \text{rk}(P) \!\leq\! K, ~ \sum\limits_{(i,j) \in [d_1] \!\times [d_2\!]} \hspace{-5mm} P_{ij} = 1 ~ \text{ and } P_{ij} \geq 0,  \ \forall (i,j) \in [d_1] \!\times\! [d_2]  \bigg\}.
\end{equation}
Assume that for some unknown $P \in \mathcal{T}_K$ we are given iid observations $X_1, \dots, X_n$ with distribution $P$, that is,  $\mathbb{P}(X_k = (i,j)) = P_{ij}$ for all $k \in [n],  (i,j)\in [d_1] \times [d_2]$. 

In this section, we consider the problem of minimax estimation of $P$ on the class $\mathcal{T}_K$ with respect to the norm $\|\cdot\|_1$.
We derive the minimax optimal rate and propose a computationally efficient estimator achieving this rate, up to a logarithmic factor.

\subsection{The localized SVD estimator}\label{subsec:discrete_Estimator}

We start by formally describing the localized SVD algorithm. In the next subsection, we will provide an intuition behind its construction and sketch the ideas of proving the upper bounds on its performance.
\vspace{2mm}

Without loss of generality, assume that the total number of observations is even and equal to $2n$.
We use sample splitting to define $H^{(1)}\in \mathbb{R}^{d_1\times d_2}$ and  $H^{(2)}\in \mathbb{R}^{d_1\times d_2}$ as the 
matrices of empirical frequencies (the histograms) corresponding to the sub-samples
$(X_1,\dots, X_n)$ and $(X_{n+1},\dots, X_{2n})$ respectively. 
In what follows, it will useful to express the matrices $P$, $H^{(1)}$ and $H^{(2)}$ in terms of their columns and rows:
\begin{align*}
    P = \Big[C_1,\dots, C_{d_2}\Big] = \left[\begin{array}{ccc}
     \!\!L_1\!\! \\
     \!\!\vdots\!\!\\
     \!\!L_{d_1}\!\!
\end{array}\right] ~~ \text{ and } ~~ H^{(\ell)} = \Big[\hatCl_1,\dots, \hatCl_{d_2}\Big] = \left[\begin{array}{ccc}
     \!\!\hatLl_1\!\! \\
     \!\!\vdots\!\!\\
     \!\!\hatLl_{d_1}\!\!
\end{array}\right], \text{ for } \ell=1,2.
\end{align*}
We set $T = \lfloor \log_2 d \rfloor -1$,  where $d\ge 2$ is a suitably chosen parameter. In this section, we take $d = d_1 \lor d_2$. Other choices of $d$ will be used when applying Algorithm~\ref{alg:Alg} as a building block for estimation of continuous densities.
For any $t \in \{0,\dots, T\}$ we define 
\begin{align}
    I_t = \bigg\{i \in [d_1]: \big\|\hatLone_i\big\|_1 \in \left(\frac{1}{2^{t+1}}, \frac{1}{2^t}\right]\bigg\}, \qquad
    J_t = \bigg\{j \in [d_2]: \big\|\hatCone_j\big\|_1 \in \left(\frac{1}{2^{t+1}}, \frac{1}{2^t}\right]\bigg\}  \label{def_I}
\end{align}
and set
\begin{align}
    I_{T+1} = \bigg\{i \in [d_1]: \big\|\hatLone_i\big\|_1 \leq \frac{1}{2^{T+1}}\bigg\}, \qquad
    J_{T+1} = \bigg\{j \in [d_2]: \big\|\hatCone_j\big\|_1 \leq \frac{1}{2^{T+1}}\bigg\}. \label{def_J_infty}
\end{align}
Next, for any $k = (t, t') \in \{0,\dots, T\!+\!1\}^2$ we define 
\begin{align}
    &U_k = I_t \times J_{t'}, \label{def_U}\\
    &\Mk = \Big(H^{(2)}_{ij}\,\mathbbm{1}_{\left\{(i,j) \in U_k\right\}}\Big)_{i,j},\label{eq:def_Mk}\qquad
    \Pk = \Big(P_{ij}\,\mathbbm{1}_{\left\{(i,j) \in U_k\right\}}\Big)_{i,j}. 
\end{align}

\begin{algorithm}[H]\label{alg:Alg}
\SetAlgoLined
\textbf{Input:} $\alpha>1$, $d\ge 2$, $N>1$, integer $n\ge 1$, matrices $H^{(1)}, H^{(2)}$.\\

\textbf{Output:} Estimator $\widehat P^*$ of $P$. 

\vspace{3mm}

\textbf{If } $n < \barc \alpha d \log N$ \, \textbf{return} $\widehat P^*=\frac{1}{2}(H^{(1)} + H^{(2)})$ 

\vspace{2mm}

\textbf{Else:} $T\leftarrow \lfloor \log_2 (d)\rfloor -1$

\vspace{2mm}

\hspace{5mm} {\textbf{For} $t,t'=0, \dots, T+1$:}
\begin{enumerate}
    \item $k \leftarrow (t,t')$; $\tau_k \leftarrow 12\sqrt{\alpha \frac{\log N}{n}2^{-t\land t'}}$ \textbf{and define $\Mk$ as in~\eqref{eq:def_Mk}}
    \item  $\hatPk \leftarrow \underset{A \in \R^{d_1\times d_2}}{\rm argmin} \left(\|\Mk - A\|_F^2 + \tau_k \|A\|_*\right)$
    \item $\widehat P \leftarrow \sum_{k \in \{0,\dots, T\!+\!1\}^2} \hatPk$
    \item $\widehat P_+ \leftarrow \big(\widehat P_{ij} \lor 0\big)_{ij} $
    \item \textbf{If} $\widehat P_+ = 0_{\,\R^{d_1\times d_2}}$ \textbf{return} $\widehat P^*=\frac{1}{2}(H^{(1)} + H^{(2)})$
\end{enumerate}
\textbf{Return} $\widehat P^*=\frac{\widehat P_+}{\|\widehat P_+\|_1}$.
\caption{Estimation procedure}
\end{algorithm}
\noindent
In what follows, ${\rm Alg1}(\alpha,d,N,n,H^{(1)},H^{(2)})$ stands for Algorithm \ref{alg:Alg} with input parameters $(\alpha,d,N,n,H^{(1)},H^{(2)})$.

\subsection{Intuition underlying the localized SVD algorithm }

Standard approaches to estimate a low-rank matrix from noisy data consist in using methods based on global SVD on the underlying matrices. 
The main drawback of such methods is that they can be sub-optimal under non-isotropic noise. In particular, it is the case for the multinomial noise that we are dealing with in our setting since $nH^{(\ell)}\sim \mathcal{M}(P,n)$ for $\ell=1,2$.  Any entry $Y_{ij}$ of a multinomial matrix $Y\sim \mathcal{M}(P,n)$ has a binomial distribution, $Y_{ij} \sim \operatorname{Bin}(n,P_{ij})$, with variance $nP_{ij}(1-P_{ij})$ that varies across the indices $(i,j)$. 
To overcome this difficulty, Algorithm \ref{alg:Alg} splits the multinomial matrix into a logarithmic number of sub-matrices, on which the multinomial noise can be more carefully controlled. Then, each sub-matrix is de-noised separately using a nuclear norm penalized estimator. Recall that this estimator is based on soft thresholding of the singular values. 
\vspace{1mm}

To appreciate why do we split the multinomial matrix as in equations~\eqref{def_I} - \eqref{def_J_infty}, assume that $Y\sim \mathcal{M}(P,n)$ and for any two subsets $I \subset [d_1]$ and $J \subset [d_2]$, consider the extractions according to $I$ and $J$: $$Y_{IJ} = (Y_{ij})_{(i,j)\in I\times J}, \quad P_{IJ} = (P_{ij})_{(i,j)\in I\times J}, \quad \text{ and } 
 \quad W_{IJ} = \frac{Y_{IJ}}{n} - P_{IJ}.$$ 
By Lemma~\ref{prop:mult_noise_2} the operator norm of $W_{IJ}$ is controlled, with high probability and ignoring the logarithmic factors and smaller order terms, by a function of column-wise and row-wise sums of entries of $P_{IJ}$:
\begin{align}\label{eq:WIJ}
    \|W_{IJ}\|^2 &\lesssim \frac{1}{n}\normsquare{P_{IJ}} = \frac{1}{n} \max\Big(\max_{i \in I} \sum_{j\in J} P_{ij},~ \max_{j \in J} \sum_{i\in I} P_{ij}\Big)\\ \nonumber
    & \leq \frac{1}{n} \max\Big(\max_{i \in I} p_i,~ \max_{j \in J} q_j\Big),
\end{align}
where $p_i = \sum\limits_{j=1}^{d_2} P_{ij}$ and $q_j = \sum\limits_{i=1}^{d_1} P_{ij}$ are the marginal probabilities. 
This bound is accurate enough if we take ``balanced" subsets $I, J$ such that 
\begin{align*}
    \forall i \in I: ~ p_i \asymp \max_{i' \in I} p_{i'}
    \quad\quad \text{ and } \quad \quad
    \forall j \in J: ~ q_j \asymp \max_{j' \in J} q_{j'},
\end{align*}
that is, we split the multinomial matrix according to similar values of $p_i$'s and $q_j$'s in order for the multinomial noise to be almost isotropic over the considered sub-matrices.

Since the values $(p_1,\dots, p_{d_1})$ and $(q_1,\dots, q_{d_2})$ are unknown we use sample splitting to obtain good enough estimators $(\hat p_1,\dots, \hat p_{d_1})$ and $(\hat q_1,\dots, \hat q_{d_2})$ from the first half of the data.
To simplify the argument, here we do not discuss these pilot estimators and assume that an oracle gives us the exact values $(p_1,\dots, p_{d_1})$ and $(q_1,\dots, q_{d_2})$.

Set $d = d_1 \lor d_2$, and for any $t \in \{0,\dots, \lfloor\log_2(d)-1\rfloor\}$, define the following index sets
\begin{equation}\label{def:index_sets}
    \widetilde I_t = \left\{i: p_i \in \left(\frac{1}{2^{t+1}}, \frac{1}{2^t}\right]\right\}, \quad \widetilde J_{t} = \left\{j: q_j \in \left(\frac{1}{2^{t+1}}, \frac{1}{2^{t}}\right]\right\}.
\end{equation}
For $T = \lfloor\log_2(d)\rfloor-1$ define 
$$\text{$\widetilde I_T = [d_1] \setminus \bigcup\limits_{t < T} \widetilde I_{t}$ and $\widetilde J_T = [d_2] \setminus \bigcup\limits_{t < T} \widetilde J_{t}$.}$$
Fix $t,t' \in \{0,\dots, T\}$ and set $k = (t,t')$,  $\widetilde M_k = \frac{1}{n} Y_{\widetilde I_t  \widetilde J_{t'}}$, $P_k = P_{\widetilde I_t  \widetilde J_{t'}}$. 
Then $\rk (P_k) \leq K$ and $\normsquare{P_k} \leq \frac{1}{2^{t\land t'}}$ by construction. 
The idea is now to take an estimator $\widehat P_k$ of $P_k$ obtained from $\widetilde M_k$ by performing a soft-thresholding of its singular values, with a threshold based on Lemma~\ref{prop:mult_noise_2}, cf.~\eqref{eq:WIJ}. 
This strategy and standard guarantees for SVD soft-thresholding (nuclear norm penalized) estimators lead to  the following bound on the error in the Frobenius norm, which holds with high probability:
\begin{align*}
    \left\| \widehat P_k-P_k\right\|_F^2 \lesssim \frac{K}{n} \normsquare{P_k} \lesssim \frac{K}{ 2^{t \land t'} n}. 
\end{align*}
Here and below, the sign $\lesssim $ hides a logarithmic factor. This implies the following bound on the $\ell_1$-error over the cell~$k$:
\begin{align}\label{eq:intuit}
    \left\| \widehat P_k-P_k\right\|_1 \lesssim \sqrt{\frac{K}{2^{t \land t'} n} |\widetilde I_t|\,|\widetilde J_{t'}|}.
\end{align}
Denoting by $\widehat P$ the $d_1\times d_2$ matrix obtained by concatenation of all the cells, and summing \eqref{eq:intuit} over $k$ we obtain 
\begin{align*}
    \|\widehat P - P\|_1 &\lesssim \sum_{t,t' = 0}^{T} \sqrt{\frac{K}{2^{t \land t'} n} |\widetilde I_t|\,|\widetilde J_{t'}|}. 
\end{align*}
By the Cauchy-Schwarz inequality and the fact that $T=\lfloor\log_2(d)\rfloor-1$ this bound can be simplified as follows:
\begin{align*}
   \sum_{t,t' = 0}^{T} \sqrt{\frac{K}{2^{t \land t'} n} |\widetilde I_t|\,|\widetilde J_{t'}|} & \le \sqrt{\sum_{t,t' = 0}^{T} \frac{K}{2^{t \land t'} n} |\widetilde I_t||\widetilde J_{t'}| \sum_{t,t'=0}^{T} 1}  \\
    & \le \log_2(d) \,\sqrt{\sum_{t,t' = 0}^{T} \frac{K}{n}\left(\frac{1}{2^{t}} + \frac{1}{2^{t'}} \right) |\widetilde I_t|\,|\widetilde J_{t'}|}  \\
    & \lesssim  \sqrt{\frac{K (d_1+d_2)}{n}},
\end{align*}
where the last inequality uses the relations 
$$\sum\limits_{t=0}^{T} |\widetilde I_t| = d_1 \quad  \text{ and } \quad \sum\limits_{t=0}^{T} 2^{-t} |\widetilde I_t| = \sum\limits_{t=0}^T \sum\limits_{i \in \widetilde{I}_t} 2^{-t} \leq \sum\limits_{t=0}^T \sum\limits_{i \in \widetilde{I}_t} 2p_i \leq 2 $$ (by the definition of $\widetilde I_t$), and the analogous relations for the family $(\widetilde J_t)_t$.

The above argument outlines a strategy for proving a bound of order $\sqrt{\frac{Kd}{n}}$ (up to a logarithmic factor) for the estimator defined by Algorithm \ref{alg:Alg}. The exact statement of the result is given in Theorems \ref{prop:UB} and \ref{th:upper discrete in expectation} below. 

\subsection{Results for discrete distributions}

We first provide minimax optimal rates for estimating a general discrete distribution on a finite set of size $D$. Without loss of generality, assume that this set is $\{1,\dots,D\}$.
Let $\Delta_D = \left\{p \in \R_+^D : \sum_{j=1}^D p_j = 1\right\}$ be the set of all probability distributions on $\{1,\dots,D\}$. 
We denote by $\mathbb{P}_p$  the probability measure of $n$ iid observations drawn from $p$, by $\mathbb{E}_p$ the  expectation with respect to $\mathbb{P}_p$, and by $\inf_{\widehat p}$ the infimum over all $\R^D$-valued estimators.

\begin{theorem}\label{th:LB_general}
Let $D\ge 1, n\ge 1$. 
There exist two absolute positive constants $c, c'$ such that
\begin{equation}\label{eq:th:LB_general1} \inf_{\widehat p} \sup_{p \in \Delta_D}  \mathbb{P}_p\left(\|\widehat p - p\|_1 \geq c \Big\{\sqrt{\frac{D}{n}}\land 1\Big\}\right) \geq c',
\end{equation}
and
\begin{equation}\label{eq:th:LB_general2}
\inf_{\widehat p} \sup_{p \in \Delta_D}  \mathbb{E}_p \|\widehat p - p\|_1 \geq c \Big\{\sqrt{\frac{D}{n}}\land 1\Big\}.    
\end{equation}
\end{theorem}
The lower bound for the expectation \eqref{eq:th:LB_general2} improves upon the bounds in~\cite{han2015minimax,kamath2015learning} that provide the same lower rate $\sqrt{\frac{D}{n}}$ under some additional conditions on $D,n$. 
The lower bound in probability \eqref{eq:th:LB_general1} is new. 

As a corollary of Theorem \ref{th:LB_general}, we get the minimax optimal rate of convergence for the problem of estimating general discrete distributions.  Given $\delta>0$ and a class of discrete distributions $\mathcal{P}\subseteq \Delta_D$, we define the {\it minimax in probability rate} of estimation of $p\in \mathcal{P}$ based on an iid sample $(X_1,\dots,X_n)$ from $p$ as 
\begin{equation*}
    \psi_\delta^*(n,\mathcal{P}) = \inf \bigg\{\psi>0 ~\Big|~ \inf_{\widehat p}\, \sup_{p \in \mathcal{P}} \mathbb{P}_p\Big( \big\|\widehat{p}(X_1,\dots,X_n)-p\big\|_1 > \psi\Big) \leq \delta\bigg\}.
\end{equation*}
%
%
\begin{corollary}\label{cor:minimax rate general discrete} Let $D\ge 1, n\ge 1$. There exist two absolute positive constants $c,c'$ such that, for all $\delta\in (0,c')$ we have
\begin{equation}\label{eq:cor:general discr1}
c \Big\{\sqrt{\frac{D}{n}}\land 1\Big\} \le \psi_\delta^*(n,\Delta_D) \le c_\delta  \Big\{\sqrt{\frac{D}{n}}\land 1\Big\},  
\end{equation}
where $c_\delta>0$ depends only on $\delta$. Furthermore,
\begin{equation}\label{eq:cor:general discr2}
\inf_{\widehat p} \sup_{p \in \Delta_D}  \mathbb{E}_p \|\widehat p - p\|_1 \asymp \sqrt{\frac{D}{n}} \land 1. 
\end{equation}
\end{corollary}
The proof of Corollary \ref{cor:minimax rate general discrete} follows immediately by combining Theorem \ref{th:LB_general} with the standard upper bound for the empirical frequency estimator
(see, for example, Lemma \ref{prop:samp_comp_TV_discrete} below).

Next, we obtain a lower bound for the class of discrete distributions $\mathcal{T}_K$ defined by probability matrices of rank at most $K$, see \eqref{def_DK}. 

\begin{theorem}[Lower bounds for $\mathcal{T}_K$]\label{prop:LB}
Let $n,K,d_1,d_2$ be positive integers. Set $d=d_1\vee d_2$. 
There exist two absolute positive constants $c, c'$ such that
\begin{equation}\label{eq:th:LBdiscr1} \inf_{\widetilde P} \sup_{P \in \mathcal{T}_K} \mathbb{P}_P\left(\|\widetilde P - P\|_1 \geq c \Big\{\sqrt{\frac{Kd}{n}}\land 1\Big\}\right) \geq c',
\end{equation}
and
\begin{equation}\label{eq:th:LBdiscr2}
\inf_{\widetilde P} \sup_{P \in \mathcal{T}_K} \mathbb{E}_P \|\widetilde P - P\|_1 \geq c \Big\{\sqrt{\frac{Kd}{n}}\land 1\Big\}.    
\end{equation}
\end{theorem}

The next two theorems give upper bounds matching the lower rate of Theorem \ref{prop:LB} up to a logarithmic factor.

\begin{theorem}[Upper bound for $\mathcal{T}_K$ in probability] \label{prop:UB}
Let $\alpha >1$, $d=d_1\vee d_2$, $N\ge 2$, and let the estimator $\widehat P^*$ be obtained by {\rm Alg\ref{alg:Alg}}($\alpha, d,  N, n , H^{(1)}, H^{(2)}$). Then there exist constants $C_0, C_1>0$ depending only on $\alpha$ such that
$$ 
\sup_{P \in \mathcal{T}_K} \mathbb{P}_P\left(\|\widehat P^* - P\|_1 > C_1 \bigg\{\sqrt{\frac{Kd}{n}} \log (d) \, \log^{1/2}(N)\land 1 \bigg\}\right) \le C_0 (\log d)^2 d N^{-\alpha}.
$$
\end{theorem}

Note that Theorem~\ref{prop:UB} can be used for several meaningful choices of $N$, such as $N=n$, $N=d$ or $N=d\vee n$.

\begin{theorem}[Upper bound for $\mathcal{T}_K$ in expectation] \label{th:upper discrete in expectation}
Let $\alpha >3/2$, $d=d_1\vee d_2$, and let the estimator $\widehat P^*$ be obtained by {\rm Alg\ref{alg:Alg}}($\alpha, d, d\vee n, H^{(1)}, H^{(2)}$). Then there exists a constant $C>0$ depending only on $\alpha$ such that
$$
\sup_{P \in \mathcal{T}_K} \mathbb{E}_P \|\widehat P^* - P\|_1 \le  C\Big\{\sqrt{\frac{Kd}{n}}(\log n)^{3/2}\land 1\Big\}.
$$
\end{theorem}

Theorem \ref{prop:LB}, Theorem \ref{prop:UB} with $N=d$ and Theorem \ref{th:upper discrete in expectation} lead to the following corollary, which provides the minimax rate for the class  $\mathcal{T}_K$. 
\begin{corollary}\label{th:main_th}
Let $\gamma > 0$ and $\delta = O(d^{-\gamma})$. There exist two constants $c,C>0$ depending only on $\gamma$ such that 
\begin{equation}\label{eq:minimax in probability} 
c\bigg\{\sqrt{\frac{Kd}{n}} \land 1\bigg\} \leq \psi_\delta^*(n,\mathcal{T}_K) \leq C \bigg\{\sqrt{\frac{Kd}{n}} (\log d)^{3/2} \land 1 \bigg\}
\end{equation}
and 
\begin{equation}\label{eq:minimax in expectation}
c\bigg\{\sqrt{\frac{Kd}{n}} \land 1\bigg\} \leq  \inf_{\widetilde P} \sup_{P \in \mathcal{T}_K} \mathbb{E}_P \|\widetilde P - P\|_1 \leq C \bigg\{\sqrt{\frac{Kd}{n}} (\log n)^{3/2} \land 1 \bigg\}.    
\end{equation}
\end{corollary}
If $K$ is substantially smaller than $d_1 \land d_2$ the rate of convergence $\sqrt{\frac{Kd}{n}}$ provided by Theorems \ref{prop:UB} and \ref{th:upper discrete in expectation}   is much faster than the  estimation rate $\sqrt{\frac{D}{n}}$ for general $D=d_1d_2$-dimensional discrete distributions. This characterizes the gain that is achieved due to the low-rank structure.

\section{Continuous distributions}\label{sec:densities}

In this section, we use the ideas developed for discrete distributions in Section~\ref{sec:discrete} to derive estimators of probability densities under the {\it generalized multi-view model}. 

We start by defining the class of considered densities. 
Let $L>0$ be a constant. For any $\beta\in(0,1]$, we say that  $f: [0,1]^2 \to \R$ is a $\beta$-H\"older function if $|f(z)-f(z')| \le L\|z-z'\|_\infty^{\beta}$ for all $z,z'\in \Supp(f)$. We denote by $\mathcal{L}_{\beta}$ the set of all $\beta$-H\"older densities supported on rectangles contained in  $[0,1]^2$: 
\begin{align}
    \mathcal{L}_{\beta} = \left\{f: [0,1]^2 \to \R \,\bigg|\, \exists \, r_1, r_2, R_1, R_2 \in [0,1] \text{ s.t.} \begin{cases} 
    \Supp(f) = [r_1,R_1] \times [r_2,R_2],\\
    f \text{ is $\beta$-H\"older on } \Supp(f),\\
    \int_{\Supp(f)} f = 1 \text{ and } f \geq 0.    \end{cases}\!\right\}\label{eq:def_L}
\end{align}
For integer $K\geq 1$, we define $\FK$ as the set of functions on $[0,1]^2$ that are sums of $K$ separable functions:
\begin{align*}
    \FK = \bigg\{(x,y)\in [0,1]^2 \longmapsto \sum_{k=1}^K u_k(x)v_k(y) \in \R ~\Big|~  u_k,v_k \in L_1[0,1], \ \forall k \in [K] \bigg\}.
\end{align*}
We consider the following set of $\beta$-H\"older probability densities 
\begin{align*}
    \mathcal{G}_{K,\beta} := \mathcal{L}_{\beta} \cap \FK.
\end{align*}
If $f \in \mathcal{G}_{K,\beta}$ we will say that $f$ follows the {\it generalized multi-view model}.
We emphasize that any function $f \in \mathcal{G}_{K,\beta}$ is only assumed to be $\beta$-H\"older on an unknown rectangle of the form $[r_1,R_1]\times [r_2,R_2]$ and not necessarily over the whole domain $[0,1]^2$. In particular, $f \in \mathcal{G}_{K,\beta}$ can have jumps at the boundary of its support $[r_1,R_1]\times [r_2,R_2]$.
Moreover, the functions $u_k$ and $v_k$ appearing in the decomposition $f(x,y) = 
\sum_{k=1}^K u_k(x) v_k(y)$ can take negative values, they need not be $\beta$-H\"older or   continuous. 
Clearly, the set $\mathcal{G}_{K,\beta}$ contains all densities that are $\beta$-H\"older on their support and can be expressed as mixtures of product densities. 

Since we consider density estimation under the $L_1$ risk it is not restrictive to assume that the support of the density is a compact set. 
Indeed, it has been highlighted in \cite{ibragimov1984more}, \cite{juditsky2004minimax}, \cite{goldenshluger2014adaptive}, \cite{chhor2021goodness} that there exist no uniformly consistent estimators with respect to the $L_1$ norm on classes of H\"older continuous densities with unbounded support. On such classes, the minimax $L_1$ rate is  of trivial order $1$ regardless of the number of observations. 
Note also that in our setting the support of $f$ is an {\it unknown} set. It allows us to handle densities that are not necessarily H\"older continuous on the whole domain $[0,1]^2$. To the best of our knowledge, this setting was not explored in the prior work.

Assume that for some integer $K \ge 1$ and some unknown $f \in \mathcal{G}_{K,\beta}$, we are given $n$ iid observations $X_1,\dots, X_n$ distributed with probability density $f$.
The goal is to estimate $f$ based on $X_1,\dots, X_n$. The minimax rate for estimation of $\beta$-H\"older densities on $[0,1]^2$ under the $L_1$ risk is known to be $n^{-\beta/(2\beta+2)}$ and this rate is attained by KDE and other basic density estimators \cite{DevroyeGyofri,DevroyeLugosi}.  
We show that the minimax rate of convergence for the class $\mathcal{G}_{K,\beta}$ with rank $K$ structure is faster than $n^{-\beta/(2\beta+2)}$ and we propose a computationally simple estimator that achieves this optimal rate up to a logarithmic factor. 

\vspace{2mm}

Our estimator is defined in Algorithm~\ref{alg:Alg_densities}. The main steps of Algorithm~\ref{alg:Alg_densities} can be summarized as follows. 
\begin{itemize}\vspace{-1mm}
\item We divide the data in two subsamples of equal size. We estimate the support $[r_1,R_1] \times [r_2,R_2]$ by the smallest rectangle of the form $[\widehat r_1, \widehat R_1]\times [\widehat r_2, \widehat R_2]$ containing all of the data points in the first subsample. 
\item We partition the estimated support $[\widehat r_1, \widehat R_1]\times [\widehat r_2, \widehat R_2]$ into disjoint rectangular cells with edges of approximately the same length $h^*$ in both dimensions. Dividing the second subsample in two equal parts, we construct two independent histogram matrices $N$ and $N'$ based on this partition. 
\item We apply Algorithm~\ref{alg:Alg} with $H^{(1)}=N$ and $H^{(2)}=N'$, which outputs a matrix with entries corresponding to the cells of the partition. We define our density estimator as a function that takes a constant value in each cell, proportional to the output of Algorithm~\ref{alg:Alg} in the cell.
\end{itemize}

Note also that we need to modify this scheme in some degenerate situations.  Indeed, the partition of the rectangle $[\widehat r_1, \widehat R_1]\times [\widehat r_2, \widehat R_2]$ degenerates if either $\widehat R_1 - \widehat r_1$ or $\widehat R_2 - \widehat r_2$ is smaller than the size of the cell $h^*$. 
If, for example, $\widehat R_1 - \widehat r_1$ is too small then, due to the $\beta$-H\"older property, $f$ does not vary much in the first coordinate direction and we are lead to a one-dimensional density estimation problem over $[\widehat r_2, \widehat R_2]$, for which we use Algorithm~\ref{alg:Alg_densities_1D}.

In Algorithm~\ref{alg:Alg_densities} we assume without loss of generality that $n$ is a multiple of $4$. For $X\in \R^2$, we denote by $\Pi_1(X)$ and $\Pi_2(X)$ its first and second coordinates, respectively: $X=(\Pi_1(X),\Pi_2(X))$. We denote by \text{Alg\ref{alg:Alg_densities_1D}}$(Z_1,\dots, Z_n,K')$ the output of Algorithm \ref{alg:Alg_densities_1D} with input $(Z_1,\dots, Z_n,K')$.


\begin{algorithm}\label{alg:Alg_densities}
\SetAlgoLined

\textbf{Input:} $X_1,\dots, X_n \in \R^2$ with $n=4k$ for an integer $k \ge 1$; $\alpha>0$; $K\in \N$; $\beta\in (0,1]$\\

\vspace{2mm}

For $m \in \{1,2\}$, set $\widehat r_m \leftarrow \min \left\{\Pi_m(X_i): i \in \{1,\dots,\frac{n}{2}\}\right\}$ and 
$\widehat R_m \leftarrow \max \left\{\Pi_m(X_i): i \in \{1,\dots,\frac{n}{2}\}\right\}$.

\vspace{2mm}

\textbf{If} $\big(\frac{K}{n}\big)^{\frac{\beta}{2\beta+1}} \log^{3/2}(n) \leq n^{-\frac{\beta}{2\beta+2}}$: set $K' = K $ ~~\textbf{ else: } set  $K' = n^{\frac{1}{2\beta+2}} $. 

\vspace{2mm}

$h^* \leftarrow ( K' /n)^{1/(2\beta+1)}$.

\vspace{2mm}

\textbf{If } $\widehat R_1 - \widehat r_1 < h^*$: 

\vspace{1mm}

\hspace{5mm} $\widehat g \leftarrow \text{Alg\ref{alg:Alg_densities_1D}}\big(\Pi_2(X_{n/2+1}),\dots, \Pi_2(X_{n}
),  K'\big)$

\vspace{1mm}

\hspace{5mm} $\widehat\phi(x,y): = \frac{1}{\widehat R_1 - \widehat r_1}\mathbbm{1}_{x \in [\widehat r_1, \widehat R_1]} \, \widehat g(y)$.

\vspace{2mm}

\textbf{Else If } $\widehat R_2 - \widehat r_2 < h^*$:

\vspace{1mm}

\hspace{5mm} $\widehat g \leftarrow \text{Alg\ref{alg:Alg_densities_1D}}\big(\Pi_1(X_{n/2+1}),\dots, \Pi_1(X_{n}
), K'\big)$

\vspace{1mm}

\hspace{5mm}  $\widehat\phi(x,y): =  \frac{1}{\widehat R_2 - \widehat r_2} \mathbbm{1}_{y \in [\widehat r_2, \widehat R_2]} \, \widehat g(x)$.

\vspace{3mm}

\textbf{Else:}

\vspace{2mm}
 
\hspace{5mm}  For $m\in\{1,2\}$, set $h_m = \ell_m^{-1}(\widehat R_m - \widehat r_m)$, where $\ell_m=\big\lfloor(\widehat R_m - \widehat r_m ) /h^*\big\rfloor$

\vspace{2mm}

\hspace{5mm} For $m\in\{1,2\}$, set $E_m = \big\{-\big\lceil \widehat r_m/h_m\big\rceil,\dots, \big\lceil (1-\widehat r_m)/h_m \big\rceil\big\}$

\vspace{3mm}

\hspace{5mm} \textbf{For} $(j,j') \in E_1 \times E_2$:

\vspace{2mm}

\hspace{10mm} $A_j \,\leftarrow \Big[\widehat r_1 + jh_1, \,\widehat r_1 + (j+1)h_1\Big)$.

\vspace{2mm}

\hspace{10mm} $B_{j'} \leftarrow \Big[\widehat r_2 + j' h_2, \,\widehat r_2 + (j'+1)h_2\Big)$.

\vspace{2mm}

\hspace{10mm} $C_{jj'} \leftarrow A_j \times B_{j'}$.

\vspace{2mm}

\hspace{10mm} $G_{jj'} \leftarrow \frac{4}{n}\sum\limits_{n/2+1\le i \le 3n/4} \mathbbm{1}_{X_i \in C_{jj'}}$; ~~ $G'_{jj'} \leftarrow \frac{4}{n}\sum\limits_{3n/4+1\le i \le n} \mathbbm{1}_{X_i \in C_{jj'}}$;

\vspace{2mm}

\hspace{5mm} \textbf{If} $\big(\frac{K}{n}\big)^{\frac{\beta}{2\beta+1}} \log^{3/2}(n) > n^{-\frac{\beta}{2\beta+2}}$:

\vspace{2mm}

\hspace{10mm} $\widehat P^* \leftarrow (G+G')/2$, where $G=(G_{jj'})_{(j,j') \in E_1 \times E_2}$, \ $G'=(G'_{jj'})_{(j,j') \in E_1 \times E_2}$


\vspace{2mm}

\hspace{5mm} \textbf{Else:}

\vspace{2mm}
 
\hspace{10mm} $\widehat P^* \leftarrow \text{Alg\ref{alg:Alg}}\Big(\alpha, |E_1| \lor |E_2|, n,  \frac{n}{4}, G, G'\Big)$

\vspace{2mm}

$\widehat\phi (x,y): = \frac{1}{ h_1 h_2}\sum\limits_{j\in E_1} \sum\limits_{j'\in E_2}  \widehat P^*_{jj'} \,\mathbbm{1}_{(x,y)\in C_{jj'}} $

\vspace{2mm} 
\textbf{If} $\widehat \phi = 0$:

\vspace{2mm}
\hspace{5mm} \textbf{Return} $\widehat f = \mathbbm{1}_{[0,1]^2}$

\vspace{2mm}

\textbf{Return} $\widehat f = \frac{\widehat\phi}{\ \ \|\widehat\phi\|_{L_1}}$

\caption{Two-dimensional density estimator}

\end{algorithm}


\begin{remark}[Choice of $E_1$ and $E_2$ in Algorithm~\ref{alg:Alg_densities}]
In order to be able to apply the results of Section~\ref{sec:discrete}, we need the union of cells $C_{jj'}$ in Algorithm~\ref{alg:Alg_densities} to be such that 
\begin{equation}\label{eq:remark1}
    \sum_{j\in E_1} \sum_{j' \in E_2}G_{jj'} = \sum_{j\in E_1} \sum_{j'\in E_2} G'_{jj'} = 1.
\end{equation}
This condition cannot be guaranteed if we only take cells that form a partition of $[\widehat r_1, \widehat R_1] \times [\widehat r_2, \widehat R_2]$ since some data points from the second subsample $\{X_{n/2+1},\dots, X_n\}$ may fall outside of the estimated support $[\widehat r_1, \widehat R_1] \times [\widehat r_2, \widehat R_2]$. Taking the sets of indices $E_1$ and $E_2$ as in Algorithm~\ref{alg:Alg_densities} ensures that the union of cells $(C_{jj'})_{j\in E_1,j'\in E_2}$ contains the whole domain~$[0,1]^2$ if $\widehat R_m -\widehat r_m>h^*$ for $m\in \{1,2\}$. In fact, under this condition, $E_1$ and $E_2$ are the sets of indices of smallest cardinality such that the union of $(C_{jj'})_{j\in E_1,j'\in E_2}$ contains $[0,1]^2$.  Of course, some of these cells $C_{jj'}$ may fall beyond $[0,1]^2$. However,
    it does not affect the sums in \eqref{eq:remark1} since $f = 0$ on such cells, so that the associated $G_{jj'}, G'_{jj'}$ vanish almost surely. 
    Moreover, the order of magnitude of $|E_1|$ and $|E_2|$ remains controlled as needed. Indeed, we have $|E_1| \lor |E_2| \leq C /h^*$, which is sufficient for our purposes. 
\end{remark}

\begin{algorithm}\label{alg:Alg_densities_1D}
\SetAlgoLined
\vspace{1mm}
\textbf{Input:} $Z_1,\dots, Z_n$ and $K'\ge 1$.\\

\vspace{3mm}

$\widehat r \leftarrow \min \left\{Z_i: i \in \{1,\dots, \lfloor \frac{n}{2}\rfloor\}\right\}$; 
$\widehat R \leftarrow \max \left\{Z_i: i \in \{1,\dots, \lfloor \frac{n}{2}\rfloor\}\right\}$.

\vspace{3mm}

$h^* \leftarrow (K'/n)^{1/(2\beta+1)}.$

\vspace{3mm}

\textbf{If } $\widehat R - \widehat r < h^*$, \textbf{then return} $\frac{1}{\widehat R - \widehat r} \mathbbm{1}_{[\widehat r, \widehat R]}$.

\vspace{2mm}

\textbf{Else:} 

\vspace{2mm}

\hspace{10mm} $h = \ell^{-1}(\widehat R - \widehat r)$, where $\ell=\big\lfloor(\widehat R - \widehat r ) /h^*\big\rfloor$

\vspace{2mm}

\hspace{10mm} $A_j = \big[\widehat r + (j-1)h,\; \widehat r + jh\big)$, for $j=1,\dots, \ell-1$, and $A_\ell=\big[\widehat r + (\ell-1)h,\; \widehat R\big]$

\vspace{2mm}

\hspace{5mm} \textbf{For} 
$j \in \big\{1,\dots, \ell \big\}$:
\vspace{2mm}

\hspace{10mm} $N_j \leftarrow \sum_{i=\lfloor \frac{n}{2}\rfloor + 1}^n \mathbbm{1}_{Z_i \in A_j}$
\vspace{2mm}

\textbf{Return} $\widehat g = \frac{1}{n-\lfloor \frac{n}{2}\rfloor}\sum\limits_{j=1}^\ell \frac{1}{h} N_j \, \mathbbm{1}_{A_j}$.\footnote{The code of Algorithms~\ref{alg:Alg_densities} and \ref{alg:Alg_densities_1D} algorithm is available at https://github.com/hi-paris/Lowrankdensity}
\caption{One-dimensional density estimator}
\end{algorithm}

\vspace{3mm}

In what follows we denote by $\mathbb{P}_f$ the probability measure induced by $(X_1,\dots,X_n)$ when $X_i$'s are iid distributed with density $f$, and by $\mathbb{E}_f$ the corresponding expectation.

\begin{theorem} \label{th:upper bound-density estimation}
There exist constants $C_0'>0$, $C_1'>0$ depending only on $\alpha$ and $L$ such that for the estimator $\widehat f$ defined by Algorithm \ref{alg:Alg_densities} with $\alpha >1$ we have
\begin{align}
  \label{eq1:th:upper bound-density estimation}
  \sup_{f\in \mathcal{G}_{K,\beta}} \mathbb{P}_f\left(\|\widehat f - f\|_{L_1} > C_1' \left\{\left(\frac{K}{n}\right)^{\beta/(2\beta+1)} \log^{3/2}\! n \land n^{-\frac{\beta}{2\beta+2}}\right\}\right) \le C_0' (\log n)^2 n^{1/(2\beta+1)-\alpha},  
  \end{align}
  and for the estimator $\widehat f$ defined by Algorithm \ref{alg:Alg_densities} with $\alpha> 4/3$ we have
  \begin{align}\label{eq2:th:upper bound-density estimation}
\sup_{f\in \mathcal{G}_{K,\beta}} \mathbb{E}_f \|\widehat f - f\|_{L_1} \le C_1' \left\{\left(\frac{K}{n}\right)^{\beta/(2\beta+1)} \log^{3/2}\! n \land n^{-\frac{\beta}{2\beta+2}}\right\}.     
  \end{align}
\end{theorem}
Theorem~\ref{th:upper bound-density estimation} guarantees that the estimator $\widehat f$ adapts to the best rate between $(K/n)^{\beta/(2\beta+1)}$, which is a ``one-dimensional" rate as function of $n$ but deteriorates as $K$ grows, and $n^{-\beta/(2\beta+2)}$, which is the standard rate of estimating a $\beta$-H\"older two-dimensional density. This demonstrates a dimension reduction property. The elbow between the two rates occurs at $K\asymp n^{1/(2\beta+2)}$. Note that, in our setting with unknown support of $f$, the possibility to estimate the density even with the slow rate $n^{-\beta/(2\beta+2)}$ does not follow from the results on nonparametric density estimation developed in the prior work (see \cite{goldenshluger2014adaptive} and the references therein).

The following lower bound shows that the rate obtained in Theorem~\ref{th:upper bound-density estimation} cannot be improved up to a logarithmic factor. We derive even a stronger lower bound that holds for the subclass of $\mathcal{G}_{K,\beta}$ containing densities with support exactly $[0,1]^2$ that can be decomposed as mixtures of separable densities that are Hölder smooth over their support. 
%
Specifically, let $\mathcal{G}_{K,\beta}^\circ $ be the set of all probability densities $f$ with support exactly $[0,1]^2$ and such that 
    \vspace{-3mm}
    \begin{equation*}
    f(x,y)= \sum_{k=1}^K w_k \, u_k(x)v_k(y), \quad \forall \ (x,y)\in [0,1]^2,
    \vspace{-3mm}
    \end{equation*}
    where $u_k , v_k $ are are $\beta$-H\"older probability densities on $[0,1]$ for all $k$, and 
        $\sum_{k=1}^K w_k = 1, \ w_k \ge 0, \forall k.$ Clearly, $\mathcal{G}_{K,\beta}^\circ\subset \mathcal{G}_{K,\beta}  $.

\begin{theorem}
\label{th:lower bound for density estimation}
Let $L>0$, $\beta\in (0,1]$. There exist two positive constants $c, c'$ that can depend only on $L$ and $\beta$ such that
\begin{equation}\label{eq1:th:lower bound for density estimation} \inf_{\widetilde f} \sup_{f \in \mathcal{G}_{K,\beta}^\circ } \mathbb{P}_f\left(\|\widetilde f - f\|_{L_1} \geq c \Big\{(K/n)^{\beta/(2\beta+1)} \land n^{-\beta/(2\beta+2)}\Big\}\right) \geq c',
\end{equation}
and
\begin{equation}\label{eq:th:LBdiscr3}
\inf_{\widetilde f} \sup_{f \in \mathcal{G}_{K,\beta}^\circ } \mathbb{E}_f \|\widetilde f - f\|_{L_1} \geq c \Big\{(K/n)^{\beta/(2\beta+1)} \land n^{-\beta/(2\beta+2)}\Big\},    
\end{equation}
where $\inf_{\widetilde f}$ denotes the infimum over all density estimators.
\end{theorem}
Consider now the minimax estimation rate on $\mathcal{G}_{K,\beta}$ with probability $\delta>0$:
\begin{equation}\label{def:minimax_rate_densities}
    \overline\psi_\delta(n,\mathcal{G}_{K,\beta}) = \inf \Bigg\{t>0 ~\Big|~ \inf_{\tilde f}\, \sup_{f \in \mathcal{G}_{K,\beta}} \mathbb{P}_f\Big( \big\|\tilde{f}(X_1,\dots,X_n)-f\big\|_{L_1} > t\Big) \leq \delta\Bigg\}.
\end{equation}
Theorems~\ref{th:upper bound-density estimation} and \ref{th:lower bound for density estimation} immediately imply the following corollary.

\begin{corollary}\label{th:densities}
For any $\gamma>0$, $\beta\in (0,1]$, $L>0$, there exist two constants $c,C>0$ depending only on $\gamma,\beta$ and $L$ such that 
$$c\big\{(K/n)^{\beta/(2\beta+1)} \land n^{-\beta/(2\beta+2)} \big\}\leq \overline\psi_{n^{-\gamma}}(n,\mathcal{G}_{K,\beta}) \leq C \big\{(K/n)^{\beta/(2\beta+1)} (\log n)^{3/2}\land n^{-\beta/(2\beta+2)} \big\}.$$ 
\end{corollary}

\section{Adaptive density estimation}

The density estimators proposed in Section~\ref{sec:densities} require knowledge of the number of separable components $K$ and of the smoothness $\beta$. In this section, we provide an estimator that is adaptive to both $K$ and $\beta$.

Throughout this section, we assume that $n\ge 3$. We consider a grid $\{\beta_1,\dots \beta_{\lceil\log (n) \log \log(n)\rceil}\}$ on the values of $\beta\in (0,1]$, where 
$$\beta_j = \Big(1+\frac{1}{\log n }\Big)^{-j+1}.$$
Let $K_{\max}\ge 2$ be an integer. 
Set $m=K_{\max}\lceil\log (n) \log \log(n)\rceil$ and fix $\alpha>4/3$. The adaptive estimator is obtained by the minimum distance choice from the family of $m$ estimators
$(\widehat f_{(K,j)})_{K\in [K_{\max}], 1\le j \le \lceil\log (n) \log \log(n)\rceil }$, where $\widehat f_{(K,j)}$ is an output of Algorithm~\ref{alg:Alg_densities} with parameters $K$, $\beta=\beta_j$, $\alpha$ when the input sample is $X_1,\dots,X_n$. The minimum distance estimator (see, e.g., \cite{DevroyeLugosi}) is defined as follows:
\begin{align}
    \widehat f^* = \widehat f_{(K^*,j^*)} \ \text{where} \ (K^*,j^*) \in \underset{(K,j)}{\rm argmin} \max_{B\in\mathcal{B}} \Big| \int_B \widehat f_{(K,j)} - \mathbb{P}_n(B) \Big|, \label{eq:def_adapt_estim}
\end{align}
and $\mathcal{B}=\{B_{ii'}, i, i'=1,\dots, m, i\ne i'\}$ with $B_{ii'}=\{x: \widehat f_{i'}(x)> \widehat f_{i}(x)\} $ for $i,i'\in \{(K,j):K\in [K_{\max}], 1\le j \le \lceil\log (n) \log \log(n)\rceil \}$. 
For a set $B\in\mathcal{B}$, the notation $\mathbb{P}_n(B)$ stands for the empirical probability measure of this set computed from the sample $(X_1,\dots,X_n)$. 

Note that each $\widehat f_{(K,j)}$ is a piece-wise constant function, so that the integrals in the definition of the adaptive estimator $\widehat f^*$ can be easily computed. To get $\widehat f^*$ we need $O(m^3)$ computations of such integrals \cite{DevroyeLugosi}, where $m$ is logarithmic in $n$.

In what follows, we set $K_{\max}=\lceil\sqrt{n}\rceil$, which is essentially the smallest sufficient choice. Indeed, $\sqrt{n} = \sup_{\beta\in (0,1]} n^{\frac{1}{2\beta+2}}$, while (as discussed in Section~\ref{sec:densities}) choosing $K$ over the threshold  $n^{\frac{1}{2\beta+2}}$ makes no sense since it does not bring any improvement compared to the standard two-dimensional rate $n^{-\frac{\beta}{2\beta+2}}$. The next theorem gives a bound on the $L_1$-risk of the adaptive estimator $\widehat f^*$.

\begin{theorem}\label{th:adaptive_rate}  
    There exists a constant $C>0$ such that estimator $\widehat f^*$ defined in~\eqref{eq:def_adapt_estim} with $K_{\max}=\lceil\sqrt{n}\rceil$ satisfies
    \begin{align*}
        \sup_{f\in \mathcal{G}_{K,\beta}} \mathbb{E}_f \|\widehat f^* - f\|_{L_1} \le C\left\{\left(K/n\right)^{\beta/(2\beta+1)} \log^{3/2}\! n \land n^{-\frac{\beta}{2\beta+2}}\right\}, \quad \forall 
        K\ge 1, \beta \in (0,1].
\end{align*}
\end{theorem}
Note that the result of Theorem \ref{th:adaptive_rate} holds for all $K\ge 1$ while the selection leading to $\widehat f^*$ is made only over the estimators $\widehat f_{(K,j)}$ with $K\le K_{\max}=\lceil\sqrt{n}\rceil$. This is because for $K> \lceil\sqrt{n}\rceil$ the estimators $\widehat f_{(K,j)}$ do not depend on $K$, cf. the definition of Algorithm \ref{alg:Alg_densities}, and they achieve the same rate as $\widehat f_{(\lceil\sqrt{n}\rceil,j)}$.

\section{Numerical experiments}

We present the results of numerical experiments on synthetic data. We have performed simulations with different values of parameters $d, n$ and the
number of components $K$. For the experiments, we
use the Python implementation of our algorithm\footnote{The code of Lowrankdensity algorithm is available at https://github.com/hi-paris/Lowrankdensity}.

Figures~\ref{fig:simul_discrete} and \ref{fig:simul_discrete_log} present numerical experiments with discrete distributions. We compare the total variation error of our estimator with that of the classical histogram estimator. 
In Figure~\ref{fig:simul_discrete}, we fix $K=1$ and $n=10^5$, and apply the two estimators on square matrices of size ranging from $d=10$ to $d=1600$. 
To better appreciate the dependency on $d$, we also represent the same experiment on a logarithmic scale in Figure~\ref{fig:simul_discrete_log}. 
We can see that the total variation error of the histogram estimator is approximately proportional to $d$, whereas the error of our estimator is approximately proportional to $\sqrt{d}$ for the ranges of values represented in this figure.


Figure~\ref{fig:simul_discrete_K_diric} presents the dependence of the total variation error on the rank $K$ for fixed dimension $d=100$ and fixed number of observations $n=10^5$. 
We also provide in Figure 4 a representation of this error on a logarithmic scale, which shows that it grows as $\sqrt{K}$. 
These two figures are obtained for low-rank matrices close to the set used in the lower bound. 


Finally, we provide simulations for the problem of density estimation (Figures~\ref{fig:simul_continuous} and~\ref{fig:simul_continuous_log}). 
We compare the standard histogram density estimator with bin width $n^{-1/4}$ and our estimator defined by Algorithm~\ref{alg:Alg_densities} with $\beta=1$, $K=1$. 
We let $n$ vary from $1000$ to $10^6$.
Again, we observe that our estimator performs better than the classical estimator and allows us to recover the one-dimensional estimation rate $n^{-1/3}$.

\vspace{2mm}

\begin{figure}
\begin{minipage}[b]{0.45\textwidth}
    \centering
    \includegraphics[width = 1\textwidth]{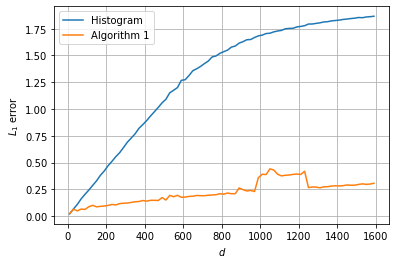}
    \caption{Total variation error of probability matrix estimators as a function of dimension $d$. 
    Here, $n=10^5$, $K=1$.
    \phantom{ }
    }
    \label{fig:simul_discrete}
\end{minipage}
\hfill 
\begin{minipage}[b]{0.45\textwidth}
    \centering
    \includegraphics[width = 1\textwidth]{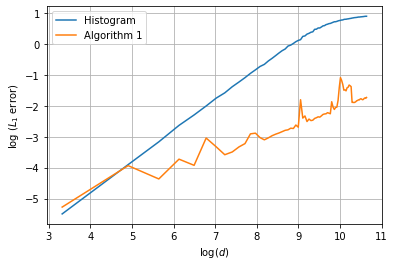}
    \caption{Total variation error of probability matrix estimators as a function of dimension $d$ (on a log-log scale). 
    Here, $n=10^5$, $K=1$.}
    \label{fig:simul_discrete_log}
\end{minipage}
\end{figure}

\begin{figure}

\begin{minipage}[b]{0.45\textwidth}
    \centering
    \includegraphics[width = 1\textwidth]{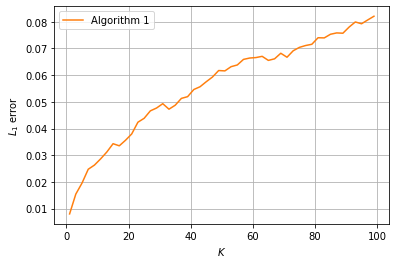}
    \caption{Total variation error of Algorithm~\ref{alg:Alg} as a function of rank $K$. Here, $d=100$ and $n=10^5$.}
    \label{fig:simul_discrete_K_diric}
\end{minipage}
\hfill
\begin{minipage}[b]{0.45\textwidth}
    \centering
    \includegraphics[width = 1\textwidth]{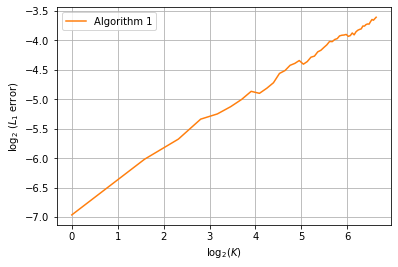}
    \caption{Total variation error of Algorithm~\ref{alg:Alg}  as a function of rank $K$ (on a log-log scale). Here, $d=100$, $n=10^5$.}
    \label{fig:simul_discrete_log_K_diric}
\end{minipage}
\end{figure}

\begin{figure}
\begin{minipage}[b]{0.45\textwidth}
    \centering
    \includegraphics[width = 1\textwidth]{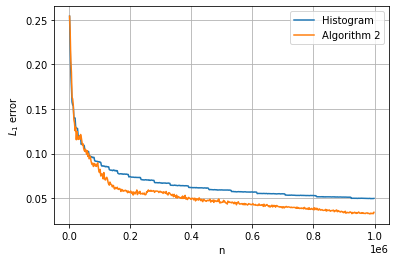}
    \caption{Total variation error of density estimators as a function of $n$ for $K=1$.}
    \label{fig:simul_continuous}
\end{minipage}
\hfill
\begin{minipage}[b]{0.45\textwidth}
    \centering
    \includegraphics[width = 1\textwidth]{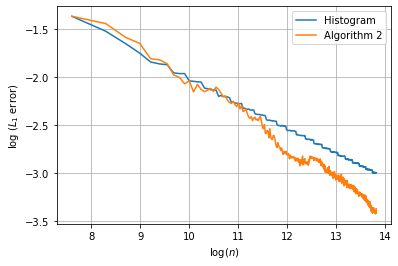}
    \caption{Total variation error of density estimators as a function of $n$ for $K=1$ (on a log-log scale).}
    \label{fig:simul_continuous_log}
\end{minipage}
\end{figure}

\section{Conclusion}

In this paper, we obtained minimax near-optimal estimators for the problem of multinomial estimation under low-rank matrix constraints and for density estimation under the generalized multi-view model. 
In both cases, we demonstrated that the rank constraint can substantially reduce the  worst-case estimation error. 
Our findings suggest that low-rank matrix model can be a powerful tool for dimension reduction in the context of density estimation. 
It would be interesting to extend our results to larger dimensions, in particular, to low-rank multinomial tensors, which  presumably needs developing different tools.


\section{Proofs}

\subsection{Upper bounds for discrete distributions}

\begin{proof}[Proof of Theorem~\ref{prop:UB}]
Let $\widehat{P}$ be the matrix defined in Algorithm \ref{alg:Alg}. 
The final estimator $\widehat{P}^*$ of Algorithm \ref{alg:Alg} is obtained from $\widehat{P}$ by the following transformation guaranteeing that $\widehat{P}^*$ is a probability matrix.
\begin{itemize}
    \item[(a)] If $\widehat P_{ij} \leq 0$ for all $(i,j) \in [d_1] \times [d_2]$ then $\widehat{P}^* = \frac{1}{2}(H^{(1)} + H^{(2)})$.
    \item[(b)] In the opposite case, if $\widehat P_{ij} >0$ for some $(i,j) \in [d_1] \times [d_2]$ then $\widehat{P}^* = \frac{\widehat P_+}{\|\widehat P_+\|_1}$, where $\widehat P_+ = \big(\widehat P_{ij} \vee 0 \big)_{ij}$.
\end{itemize}
    We first observe that it suffices to prove Theorem~\ref{prop:UB} for the estimator $\widehat{P}$ instead of $\widehat{P}^*$ since for any probability matrix $P$ the following inequality holds:  
\begin{equation}
 \label{reduction to probability}
\|P \!-\! \widehat P^*\|_1 \le 2  \|P \!-\! \widehat P\|_1.   
\end{equation}
Indeed, in case (a) this inequality is satisfied since  $\|\widehat{P}^* - P\|_1\leq 
\frac{1}{2}(\|H^{(1)}\|_1 + \|H^{(2)}\|_1)+\|P\|_1= 2$, while $\|\widehat{P} - P\|_1\geq \|P\|_1 =1$.
In case (b) we have
    \begin{align*}
    \big\|P \!-\! \widehat P^*\big\|_1 &\leq \big\|P - \widehat P_+\big\|_1 + \big\|\widehat P_+ \!-\! \widehat P^*\big\|_1 \leq \big\|P \!-\! \widehat P\big\|_1 + \bigg\|\widehat P_+\Big(1 \!-\! \frac{1}{\big\|\widehat P_+\big\|_1}\Big)\bigg\|_1\\
    & = \big\|P \!-\! \widehat P\big\|_1 +\left|\big\|\widehat P_+\big\|_1 \!-\!1 \right| \leq \big\|P \!-\! \widehat P\big\|_1 + \big\|P \!-\! \widehat P_+\big\|_1\\
    &\leq 2  \big\|P \!-\! \widehat P\big\|_1.
\end{align*}
Also, in both cases (a) and (b) we have $\|\widehat{P}^* - P\|_1\leq 2$. Therefore, to prove the theorem it suffices to show that  $\|\widehat P \!-\! P\|_1 \leq C_1 \sqrt{\frac{Kd}{n}}  \log (d) \,\log^{1/2}(N)$ with probability at least $1 -  C_0 (\log d)^2 dN^{-\alpha}$, where the constants $C_0,C_1>0$ depend only on $\alpha$. 
 
Consider first the case $n < \barc \alpha d \log ( N)$. Then Algorithm \ref{alg:Alg} outputs the estimator $\widehat P^* = \frac{1}{2}(\Yone + \Ytwo)$ satisfying $ \|\widehat P^* - P\|_1 \leq 2 = 2 \land C\sqrt{\frac{Kd}{n}} \log (d) \,\log^{1/2}(N)$ for a constant $C>0$ depending only on $\alpha$. Thus, the result of the theorem is granted if $n < \barc \alpha d \log ( N)$, and it remains to consider the case $n \geq \barc \alpha \, d \log ( N)$.

\medskip

For the rest of this proof, we assume that $n \geq \barc \alpha \, d \log ( N)$.
We recall the notation of Section~\ref{subsec:discrete_Estimator} and define the random event
\begin{align*}
    \mathcal{A} & = \left\{
    \begin{array}{lll}
          \|\hatLone_i\|_1 \geq \frac{1}{4} \|L_i\|_1 \ \text{ for all } i \in [d_1]\, \text{ such that } \|L_i\|_1 \geq \barc \,\alpha\frac{\log  N}{n} \\
         \text{and} \\
         \|\hatCone_j\|_1 \geq \frac{1}{4} \|C_j\|_1 \ \text{ for all } j \in [d_2]\, \text{ such that } \|C_j\|_1 \geq \barc\,\alpha\frac{\log  N}{n} 
    \end{array}
    \right\}. 
\end{align*} 
Let us evaluate the probability $\mathbb{P}(\mathcal{A})$. We will use the following lemma proved in Section \ref{sec:auxiliary}.

\begin{lemma}\label{lem_concent_bin}
    Let $Y \sim \M(P,n)$ be a multinomial $d_1\times d_2$ random matrix. 
    To any column $j \in [d_2]$, we associate a subset of row indices $V_j \subseteq [d_1]$ and a random variable $Z_j = \sum_{i \in V_j} Y_{ij}$. 
    We define $\lambda_j = \sum_{i\in V_j} P_{ij} = \frac{1}{n}\mathbb E(Z_j)$. 
    Let $\alpha>0$, $N>1$ be such that $n>\barc\alpha \log (N)$. 
    If $\lambda_j \in \big[\frac{\barc\alpha\log (N)}{n} ,1\big)$ for any $j\in J$, where $J\subseteq [d_2]$, then
    $$
    \mathbb{P}\left(\forall j \in J: \frac{Z_j}{n}\ge \frac{\lambda_j}{4}\right)
    \geq 1- \frac{|J|+1}{N^\alpha}.
    $$
\end{lemma}

By applying Lemma~\ref{lem_concent_bin} with  $V_j=[d_1]$ for all $j$, $Z_j = n \|\hatCone_j\|_1$, $\lambda_j=\|C_j\|_1=\sum_{i=1}^{d_1} P_{ij}$, and $J=\left\{j\in [d_2]:\, \|C_j\|_1 \geq \barc\,\alpha\frac{\log  N}{n} \right\}$ we obtain
\begin{align*}
\mathbb{P}\left(\|\hatCone_j\|_1 \geq \frac{1}{4} \|C_j\|_1 \ \text{ for all } j \in [d_2]\, \text{ such that } \|C_j\|_1 \geq \barc\,\alpha\frac{\log  N}{n} \right)\ge 1 - \frac{d_2+1}{N}^{\alpha}.
\end{align*}
Quite analogously,
\begin{align*}
\mathbb{P}\left(\|\hatLone_i\|_1 \geq \frac{1}{4} \|L_i\|_1 \ \text{ for all } i \in [d_1]\, \text{ such that } \|L_i\|_1 \geq \barc\,\alpha\frac{\log N}{n} \right)\ge 1 - \frac{d_1+1}{ N}^{\alpha}.
\end{align*}
Therefore, since $d=d_1\vee d_2\ge 2$,
\begin{equation}\label{eq:proba_A1}
    \mathbb{P}(\mathcal{A}) \ge 1 - \frac{d_1+d_2+2}{N}^{\alpha}\geq 1 - 3\,d N^{-\alpha}. 
\end{equation}
Recall that $T= \lfloor\log_2(d)\rfloor-1$ and note that, for $n \geq \barc \alpha \, d \log (N)$,
\begin{align*}
    \forall t \in \{0,\dots, T\!+\!1\}: \frac{1}{2^{t-1}} \geq \frac{1}{d} \geq \barc \alpha \frac{\log  N}{n}.
\end{align*}
Fix $k = (t,t') \in \{0,\dots, T\!+\!1\}^2$ and $(i,j) \in U_k$.
Recalling the definition of  $U_k$ in \eqref{def_U} we get $\|\hatLone_i\|_1 \leq 2^{-t}$, so that on the event $\mathcal{A}$ we have 
$\|L_i\|_1 \leq  2^{2-t} \lor  \barc \alpha \frac{\log  N}{n} = 2^{2-t}$, and similarly, $\|C_j\|_1 \leq  2^{2-t'}$. 
By definition of $U_k$, we therefore have $\normsquare{\Pk} \leq 2^{2-t\land t'}$ on the event $\mathcal{A}$. 
Moreover, $\Pk$ has rank at most $K$. 

For any fixed $\Yone\in \mathcal{A}$, we apply Lemma~\ref{prop:mult_noise_2} with  $Q = \Pk$ and $W_Q = W^{(k)}:=\Mk-\Pk$. This yields that, with probability at least $1-{2d} N^{-\alpha}$, we have $\|W^{(k)}\|^2 \leq 9 \max\Big\{ \normsquare{\Pk}\frac{\alpha \log  N}{n} , \Big(\frac{\alpha\log N}{n}\Big)^2\Big\} \leq 9 \alpha\frac{\log  N}{n}2^{2-t\land t'}$ for any fixed $\Yone\in \mathcal{A}$. By the definition of Algorithm~\ref{alg:Alg}, 
\begin{align*}
    \hatPk = \arg \min_{A \in \R^{d_1\times d_2}} \|\Mk - A\|_F^2 + \tau_k \|A\|_*,
\end{align*}
where $\tau_k = 6\sqrt{\alpha \frac{\log N}{n}2^{2-t\land t'}}$. 
Standard deterministic guarantees for nuclear norm penalized estimators (see, for example, \cite[Lemma 6.9]{giraud2021introduction}) imply that if $\|W^{(k)}\| \leq \tau_k/2$ and $\Pk$ has rank at most $K$, then there exists an absolute constant $c^*>0$ such that $\|\hatPk - \Pk\|_F^2 \leq c^* K \tau_k^2 $. Therefore, we have
\begin{align*}
    \|\hatPk - \Pk\|_F^2 & 
    \leq \frac{C}{n}\frac{K \log  N}{2^{t\land t'-2}},
\end{align*}
where the constant $C>0$ depends only on $\alpha$ and this inequality holds with probability at least $1-2d  N^{-\alpha}$ for any fixed $\Yone\in \mathcal{A}$.

Now, note that since the entries of $\Yone$ are non-negative and sum to $1$ the definitions  of $I_t$ and $J_t$ in \eqref{def_I} - \eqref{def_J_infty} imply that $|I_t| \leq 2^{t+1} \land d$ and $|J_{t'}| \leq 2^{t'+1} \land d$. Thus, for $U_k$ defined in \eqref{def_U} we have  $|U_k| \leq \left(2^{t+1} \land d\right)(2^{t'+1} \land d)$. We use this bound to link the $\ell_1$-norm of $\hatPk - \Pk$  to its Frobenius norm:
\begin{align*}
    \|\hatPk - \Pk\|_1 &\leq \sqrt{|U_k|}\|\hatPk - \Pk\|_F \leq \sqrt{2\left(2^t \land d\right)\left(2^{t'} \land d\right)} \sqrt{\frac{C}{n}\frac{K \log  N}{2^{t\land t'-2}}} \nonumber \\
    & \le \sqrt{(2^{t\lor t'} \land d)\frac{8 C K \log  N}{n}}.
\end{align*}
Using the union bound we get that, for any fixed $\Yone\in \mathcal{A}$, the estimator
$
    \widehat P = \sum_{k \in \{0,\dots, T\!+\!1\}^2} \hatPk
$
is such that, with probability at least $1-2(T+2)^2 d  N^{-\alpha}$,
\begin{align*}
    \|\widehat P - P\|_1 &= \sum_{k \in \{0,\dots, T\!+\!1\}^2} \|\hatPk - \Pk\|_1 
     \leq  \sqrt{\frac{8 C K \log  N}{n}} \sum_{(t,t') \in \{0,\dots, T\!+\!1\}^2}  2^{(t\lor t')/2}.
\end{align*}
In view of \eqref{eq:proba_A1}, this bound holds with probability over the joint distribution of $(\Yone,\Ytwo)$, which is at least $1 -(2(T+2)^2 + 3)d  N^{-\alpha}\ge 1-C_0 (\log d)^2d  N^{-\alpha}$
for an absolute constant $C_0>0$. Note that
\begin{align*}
    &\sum_{(t,t') \in \{0,\dots, T\!+\!1\}^2}  2^{(t\lor t')/2}  \le 2 \sum_{(t,t') \in \{0,\dots, T\!+\!1\}^2: t\ge t'}  2^{t/2} = 2\sum_{t=0}^{T+1}{(t+1)} \,2^{t/2}
   \\
 &\qquad \le  2\int_0^{T+1} ( x+1)\, 2^{x/2} dx = 2 \Big(\frac{2}{\log 2}\Big)^2  [2^{(T+1)/2}\Big(\frac{\log 2}{2}(T+1)-1\Big){+1}]+ 2\frac{2}{\log 2}\Big(2^{(T+1)/2}-1\Big)
    \\
   &\qquad  
   = 2^{(T+1)/2} \left(\frac{4}{\log 2}(T+1)+\frac{4}{\log 2} - \frac{8}{(\log 2)^2} \right) - \left(\frac{4}{\log 2} - \frac{8}{(\log 2)^2} \right) \le \frac{4\sqrt{d}\log_2(d)}{\log 2}.
\end{align*}
It follows that $\|\widehat P - P\|_1 \le C_1 \sqrt{\frac{Kd}{n}} \log(d)  \log^{1/2}(N)$ with probability at least
$1-C_0 (\log d)^2d  N^{-\alpha}$, where $C_0>0$ is an absolute constant and $C_1>0$ depends only on $\alpha$.
\end{proof}

\begin{proof}[Proof of Theorem \ref{th:upper discrete in expectation}] 
We follow the same argument as in the proof of Theorem~\ref{prop:UB} with the only difference that now we set $N=d\vee n$ in Lemmas~\ref{lem_concent_bin} and~\ref{prop:mult_noise_2}.  This leads to the bound
\begin{align*}
    \|\widehat P^* - P\|_1 &
    \leq C_1\sqrt{\frac{Kd}{n}} \log^{3/2}(d\vee n), 
\end{align*}
 which holds with probability at least $1 - C_0 (\log (d\vee n))^2d(d\vee n)^{-\alpha}$. Therefore, since $\|\widehat P^* - P\|_1\le 2$ and $\alpha>3/2$ we have
 \begin{align*}
    \E_P\|\widehat P^* - P\|_1 &
    \leq C_1\sqrt{\frac{Kd}{n}} \log^{3/2}(d\vee n) + 2C_0 (\log (d\vee n))^2 n^{1-\alpha} \le C\sqrt{\frac{Kd}{n}} \log^{3/2}(d\vee n)
\end{align*}
 for some constant $C>0$ depending only on $\alpha$. Noticing that $\sqrt{\frac{Kd}{n}} \log^{3/2}(d\vee n)\wedge 1= \sqrt{\frac{Kd}{n}} \log^{3/2}(n)\wedge 1$ for $d,n\ge 2$ completes the proof.
\end{proof}


\subsection{Lower bounds for discrete distributions} \label{sec:LB_discrete}

The aim of this subsection is to prove Theorems~\ref{th:LB_general} and \ref{prop:LB}.  Note that it suffices to prove Theorem~\ref{prop:LB}. Indeed,
Theorem~\ref{th:LB_general} is obtained as a corollary of Theorem~\ref{prop:LB} by taking $K=d_1=1$ and $d_2=D$. 

\medskip

\noindent {\it Proof of Theorem~\ref{prop:LB}.} We first prove the lower bound in expectation \eqref{eq:th:LBdiscr2} and then combine it with Lemma \ref{lem:reduction of lower bound} (see Section \ref{sec:auxiliary}) to deduce the bound in probability \eqref{eq:th:LBdiscr1}. 

\begin{proof}[Proof of \eqref{eq:th:LBdiscr2}] Note first that the result is trivial if $d_1=d_2=1$. Therefore, assume that $d_2\ge2$ 
and, without loss of generality, $d_2 \geq d_1$. Set $D_2 = \lfloor d_2/2 \rfloor$ and $D = 2 K D_2$. Define $\gamma = \mbox{\fontsize{13pt}{12pt}\(\left(\frac{1}{4\sqrt{n D}}\land \frac{1}{2D}\right)\)}$. 
For any $\epsilon=(\epsilon_{ij}) \in \{-1,1\}^{K \times D_2}$, define a $d_1\times d_2$ matrix $P_\epsilon$ with the following entries:
\begin{equation}\label{def_perturbed_mat_expect}
    \forall (i,j) \in \big[d_1\big] \times \big[d_2\big]: ~ P_\epsilon(i,j) =  \begin{cases}
    \mbox{\fontsize{13pt}{12pt}\(\frac{1}{D}\)} + \epsilon_{ij}\gamma & \text{ if } i \leq K \text{ and } j \leq D_2, \vspace{1mm}\\  \mbox{\fontsize{13pt}{12pt}\(\frac{1}{D}\)} - \epsilon_{i(j-D_2)} \gamma & \text{ if } i \leq K \text{ and } D_2 < j \le 2 D_2, \\  0 & \text{ otherwise.}\end{cases}
\end{equation}
Consider the set of $d_1\times d_2$ matrices 
\begin{equation}\label{Hypotheses_LB_expectation}
\mathcal{P}: = \left\{P_\epsilon ~\big|~ \epsilon \in \{-1,1\}^{K \times D_2} \right\}.
\end{equation}
This set consists of $2^m$ matrices, where $m = KD_2$. Note also that $\mathcal{P}\subset \mathcal{T}_K$. Indeed,  all matrices $P \in \mathcal{P}$ are of rank $K$ and have non-negative entries summing up to $1$. 
For any $\epsilon, \epsilon' \in \{-1,1\}^{K \times D_2}$ we have
\begin{align}\label{TV_Hamming}
     \|P_\epsilon - P_{\epsilon'}\|_1 = 2\gamma \rho(\epsilon, \epsilon'), 
\end{align}
where $\rho(\epsilon,\epsilon') = \sum\limits_{i=1}^K \sum\limits_{j=1}^{D_2} \mathbbm{1}_{\epsilon_{i,j} \neq \epsilon'_{i,j}}$ denotes the Hamming distance between $\epsilon$ and $\epsilon'$. 

We now apply Assouad's lemma (see Theorem 2.12$(iv)$ in \cite{tsybakov2009introduction}). Let 
$\epsilon, \epsilon' \in \{-1,1\}^{K \times D_2}$ be such that $\rho(\epsilon,\epsilon') = 1$. Denote by $(i_0, j_0)$, where $i_0 \in [K]$ and $j_0 \in [D_2]$, the unique pair 
of indices such that $\epsilon_{i_0, j_0} = -\epsilon'_{i_0, j_0}$. 
Then the $\chi^2$-divergence between $P_\epsilon$ and $P_{\epsilon'}$ satisfies
\begin{align*}
    \chi^2\left(P_\epsilon, P_{\epsilon'}\right) &= \sum_{i=1}^K \sum_{j=1}^{2D_2} \frac{\left(P_{\epsilon}(i,j) - P_{\epsilon'}(i,j)\right)^2}{P_{\epsilon'}(i,j)} \\
    &= \frac{\left(P_{\epsilon}(i_0,j_0) - P_{\epsilon'}(i_0,j_0)\right)^2}{P_{\epsilon'}(i_0,j_0)} + \frac{\left(P_{\epsilon}(i_0,j_0+D_2) - P_{\epsilon'}(i_0,j_0+D_2)\right)^2}{P_{\epsilon'}(i_0,j_0+D_2)} \\
    & \leq  \frac{8\gamma^2}{\frac{1}{D} - \gamma}  \leq 16 \gamma^2 D ~~ \Big(\text{since} \ \gamma \leq \frac{1}{2D}\Big)\\
    & = \frac{1}{n} \land \frac{4}{D} \leq  \frac{1}{n}.
\end{align*}
By Lemma 2.7 in \cite{tsybakov2009introduction}, the $\chi^2$-divergence between the corresponding product measures satisfies 
\begin{align}\label{TV_leq_chi2}
  \chi^2\left(P_\epsilon^{\otimes n}, P_{\epsilon'}^{\otimes n}\right)  = \left(1+\chi^2\left(P_\epsilon, P_{\epsilon'}\right)\right)^n -1 \le  e-1
\end{align}
for all $\epsilon, \epsilon' \in \{-1,1\}^{K \times D_2}$ such that $\rho(\epsilon,\epsilon') = 1$.
Taking into account \eqref{TV_Hamming}, \eqref{TV_leq_chi2}, arguing as in  \cite[Example 2.2]{tsybakov2009introduction} we obtain and applying \cite[Theorem 2.12$(iv)$]{tsybakov2009introduction} we obtain
\begin{align}\nonumber
\inf\limits_{\widetilde P} \max\limits_{P \in \mathcal{P}} \E_P \|\widetilde P - P\|_1 &\geq 
\frac{\gamma m}{4} \exp(1-e) = \frac{D}{8} \exp(1-e) \left(\frac{1}{4\sqrt{n D}}\land \frac{1}{2D}\right)
\\
\label{eq:proof-lower-0}
& \ge \frac{\exp(1-e)}{32} \Big\{ \sqrt{\frac{D}{n}} \land  1 \Big\}
\ge c \Big\{ \sqrt{\frac{Kd_2}{n}} \land  1 \Big\},
\end{align}
where $c>0$ is an absolute constant. This proves \eqref{eq:th:LBdiscr2}.
\end{proof}

\begin{proof}[Proof of \eqref{eq:th:LBdiscr1}]
We apply Lemma \ref{lem:reduction of lower bound}, where we take $\mathcal{P}_0$ as the set of all $d_1\times d_2$ matrices, $\mathcal P$ as the set of matrices defined in \eqref{Hypotheses_LB_expectation}, and we consider the metric $v(P,P')=\|P - P'\|_1$. Notice that, under these definitions, assumption \eqref{eq:lem:reduction of lower bound1} of Lemma \ref{lem:reduction of lower bound} is satisfied with $U = \left[\frac{1}{D} \,\mathbbm{1}_{i \leq K, j\le 2D_2 }\right]_{i,j}$ and $s=\gamma D=\mbox{\fontsize{13pt}{12pt}\(\Big(\frac14\sqrt{\frac{D}{n}}\land \frac{1}{2}\Big)\)}$, where $D,D_2$ are defined in the proof of \eqref{eq:th:LBdiscr2}. Moreover, due to \eqref{eq:proof-lower-0} there exists an absolute constant $a>0$ such that assumption \eqref{eq:lem:reduction of lower bound2} of Lemma \ref{lem:reduction of lower bound} is satisfied with the same $s$. Thus, we can apply Lemma~\ref{lem:reduction of lower bound}, which yields the desired lower bound in probability.
\end{proof}


 \subsection{Upper bounds for continuous distributions}

\begin{proof}[Proof of Theorem~\ref{th:upper bound-density estimation}]

 Recall that in Algorithm~\ref{alg:Alg_densities} we assume that $n$ is a multiple of $4$. Note first that it suffices to consider the case $\beta > -\frac{2 \log((8L)^{-3/2} \land 1)}{ \log(n)}$. Indeed, if the interval $(0, -\frac{2 \log((8L)^{-3/2} \land 1)}{ \log(n)}]$ is non-empty and $\beta$ belongs to this interval then
the desired rate from equation~\eqref{eq1:th:upper bound-density estimation} satisfies, for $n\ge4$,
\begin{align*}
    &\left(\frac{K}{n}\right)^{\beta/(2\beta+1)} \log^{3/2}\! n \land n^{-\frac{\beta}{2\beta+2}} 
    \geq \left(\frac{1}{n}\right)^{\beta} \log^{3/2}\! n \land n^{-\frac{\beta}{2}} \\
    & \qquad \geq  ((8L)^{-3/2} \land 1)^2 \log^{3/2}4 \land ((8L)^{-3/2} \land 1) 
     =: a(L).
\end{align*}
On the other hand, since  $\widehat f$ is a probability density we have 
the trivial bound $\|f - \widehat f\|_{L_1}\le 2$ for all probability densities $f$.
Thus, we immediately get~\eqref{eq1:th:upper bound-density estimation} with any $C'_1 > 2/a(L)$.

For the rest of this proof, we assume that $\beta>-\frac{2 \log((8L)^{-3/2} \land 1)}{ \log(n)}$.

  Fix a density $f \in \mathcal G_K$, and consider the marginal densities $g_1:[0,1] \longrightarrow \R$ and $g_2:[0,1] \longrightarrow \R$ defined by 
$$g_1(x) = \int_0^1 f(x,y) \, dy \quad \text{ and } \quad g_2(y) = \int_0^1 f(x,y) \, dx, \qquad \forall x,y \in [0,1].$$
The functions $g_1$ and $g_2$ are $\beta$-H\"older on $[r_1,R_1]$ and $[r_2,R_2]$, respectively. They are 
densities of random variables $\Pi_1(X_i)$ and $\Pi_2(X_i)$, respectively, where $\Pi_j(\cdot)$ denotes the projector onto the $j$-th coordinate. Let $q_-^{(j)}$ and $q_+^{(j)}$ be the quantiles  of order $n^{-1/3}$ and $1-n^{-1/3}$ of the probability measure induced by $g_j$:
    $$ \int_{-\infty}^{q_-^{(j)}}g_j = n^{-1/3} \quad \text{ and } \quad \int_{q_+^{(j)}}^{+\infty}g_j = n^{-1/3}. $$
Since $f \in \mathcal G_K$ it follows from the definition in~\eqref{eq:def_L} that there exist real numbers $r_1, r_2, R_1, R_2 \in [0,1]$ (depending on $f$) such that $\Delta_j:=R_j-r_j>0$, $j=1,2$, and 
\begin{align*}
    \begin{cases} 
    \Supp(f) = [r_1,R_1] \times [r_2,R_2],\\
    f \text{ is $\beta$-H\"older over } \Supp(f),\\
    \int_{\Supp(f)} f = 1 \text{ and } f \geq 0.    \end{cases}
\end{align*}
For $j=1,2$, let $\widehat g_j$ denote the output of Alg~\ref{alg:Alg_densities_1D}($\Pi_j(X_{n/2+1}),\dots, \Pi_j(X_n),  K')$. Let $\widehat r_j$ and $\widehat R_j$ be the estimators of~$r_j$ and $R_j$ defined in Algorithm~\ref{alg:Alg_densities}.
For the rest of this proof, we place ourselves on the event $\mathcal{E} \cap \mathcal{F}$, where
\begin{align*}
    \mathcal{E} &= \left\{\widehat r_j < q_-^{(j)} < q_+^{(j)} < \widehat R_j \ \text{for} \ j=1,2\right\},\\[5pt]
    \mathcal{F} &= \left\{\|g_j - \widehat g_j\|_{L_1} \leq C( K'/n)^{\beta/(2\beta+1)} \ \text{for} \ j=1,2\right\}.
\end{align*}
Here, $C>0$ is the constant from Lemma~\ref{lem:density_1D} and {$K' = K $ if $\big(\frac{K}{n}\big)^{\frac{\beta}{2\beta+1}} \log^{3/2}(n) \leq n^{-\frac{\beta}{2\beta+2}}$, otherwise  $K' = n^{\frac{1}{2\beta+2}} $. }
By Lemma~\ref{lem:density_1D}, if $n\ge 64$ and $ K'/n\le (8L)^{-3/2}\wedge 1$ then $\mathbb P(\mathcal{F}) \geq 1 - 10\exp\big(\!-n^{1/3}\big)$, and $\mathbb P(\mathcal{E}) \geq 1-4\exp(-n^{2/3}/2)$, so that $\mathbb P(\mathcal{E} \cap \mathcal{F}) \geq 1 - 14 \exp\big(\!-n^{1/3}\big)$. 
Here, the condition $K'/n \le (8L)^{-3/2}\wedge 1$ is satisfied because, by the definition of $K'$, we have  $K' \leq n^{\frac{1}{2\beta+2}} \leq n^{1/2} \le n ( (8L)^{-3/2}\wedge 1)$, where the last inequality is due to the fact that $1\ge \beta>-\frac{2 \log((8L)^{-3/2} \land 1)}{\log(n)}$.

\vspace{1mm}

Let $\widehat \Delta_j=\widehat R_j-\widehat r_j$, $j=1,2$.  We distinguish between the following two cases.
\vspace{1mm}

\textbf{First case:} $\widehat \Delta_1 \land \widehat\Delta_2 \leq h^*$. 
It suffices to assume that $\widehat \Delta_1 \leq h^*$ since the case $\widehat \Delta_2 \leq h^*$ is treated in the same way.
If $\widehat \Delta_1 \leq h^*$ the estimator $\widehat \phi$ in Algorithm~\ref{alg:Alg_densities} has the form $\widehat \phi (x,y) = \frac{1}{\widehat \Delta_1}\mathbbm{1}_{x \in [\widehat r_1, \widehat R_1]} \widehat g_2(y)$ and we get  
\begin{align*}
    \|f - \widehat \phi\|_{L_1} &\leq \int_{y=0}^1 \int_{x=r_1}^{R_1} \left|f(x,y) - \frac{g_2(y)}{\Delta_1}\right| dx dy + \int_{y=0}^1 \int_{x=r_1}^{R_1}  \left|\frac{g_2(y)}{\Delta_1} - \widehat \phi(x,y)\right| dx dy
    \\[5pt]
    & \leq \int_{y=0}^1 \int_{x=r_1}^{R_1}  \Bigg\{\frac{1}{\Delta_1}\int_{r_1}^{R_1} \underbrace{\left|f(x,y) - f(x',y)\right|}_{\leq L\Delta_1^{\beta}}dx'\Bigg\} dx dy + \int_{y=0}^1 \int_{x=r_1}^{R_1}  \left|\frac{g_2(y)}{\Delta_1} - \frac{\widehat g_2(y)}{\widehat \Delta_1}\mathbbm{1}_{x\in[\widehat r_1, \widehat R_1]} \right| dx dy
    \\[5pt]
    & \leq \int_{y=0}^1 \int_{x=r_1}^{R_1}L\Delta_1^{\beta} \, dxdy ~ 
    + \int_{y=0}^1 \left|g_2(y)\frac{\widehat \Delta_1}{\Delta_1} - \widehat g_2(y) \right| dy 
    + \int_{y=0}^1 \int_{x \in [r_1,R_1]\setminus [\widehat r_1,\widehat R_1]} \frac{g_2(y)}{\Delta_1} dxdy
    \\[5pt]
    & \leq L\Delta_1^{\beta+1} + \|g_2 - \widehat g_2\|_{L_1} + \|\widehat g_2 \|_{L_1} \left|1-\frac{\widehat \Delta_1}{\Delta_1} \right| + \|g_2\|_{L_1} \frac{\Delta_1 - \widehat \Delta_1}{\Delta_1}
    \\[5pt]
    & \leq L(h^*)^{\beta+1}\Big(1+ 16 n^{-1/(2\beta+1)}\Big)^{\beta+1} + C\Big(\frac{ K'}{n}\Big)^{\beta/(2\beta+1)} + 32 n^{-1/(2\beta+1)}
    \\[5pt]
    & \leq C'\Big(\frac{K'}{n}\Big)^{\frac{\beta}{2\beta+1}}\\
    & = C' \begin{cases}
        (K/n)^{\frac{\beta}{2\beta+1}} & \text{ if } (K/n)^{\frac{\beta}{2\beta+1}} \log^{3/2}(n) \leq n^{-\frac{\beta}{2\beta+2}}\\[5pt]
       n^{-\frac{\beta}{2\beta+2}}& \text{ otherwise}
    \end{cases} \\[5pt]
    &  \leq C'\Big(\frac{K}{n}\Big)^{\beta/(2\beta+1)} \log^{3/2}(n) \land n^{-\frac{\beta}{2\beta+2}}.
\end{align*}
where the constant $C'>0$ depends only on $L$. Here, we have used the facts that, by Lemma~\ref{lem:density_1D}, if $\mathcal E$ holds and $\widehat \Delta_1 \leq h^*$, then $\Delta_1 - \widehat \Delta_1 \leq 16\widehat \Delta_1 n^{-1/(2\beta+1)} \leq 16 \Delta_1 n^{-1/(2\beta+1)}$  and on the event $\mathcal{F}$ we have $\|g_2 - \widehat g_2\|_{L_1} \leq C ( K'/n)^{\beta/(2\beta+1)}$. We conclude that, in the first case, the bound \eqref{eq1:th:upper bound-density estimation} of Theorem \ref{th:upper bound-density estimation} is satisfied. 

\vspace{2mm}

\textbf{Second case:} $\widehat \Delta_1 > h^*$ and $\widehat \Delta_2>h^*$. 
In this case the estimator $\widehat \phi$ in Algorithm~\ref{alg:Alg_densities} has the form
$$
\widehat\phi (x,y) = 
    \dfrac{1}{ h_1 h_2}\displaystyle\sum_{j\in E_1} \displaystyle\sum_{j'\in E_2}  \widehat P^*_{jj'} \,\mathbbm{1}_{(x,y)\in C_{jj'}}. 
$$
Recalling the notation of Algorithm~\ref{alg:Alg_densities} we have
\begin{align}\label{eq:b}
    b := |E_1| \lor |E_2| \le 1 + 2\big\lceil h_1^{-1} \big \rceil \lor 2\big\lceil h_2^{-1} \big \rceil \leq C(n/K')^{1/(2\beta+1)}, 
\end{align}
where $C>0$ is an absolute constant.
Since {$f \in \mathcal G_{K,\beta}$} we have the representation $f(x,y)= \sum_{k=1}^K  u_k(x) v_k(y)$ with some functions $u_k \in L_1[0,1]$, $v_k\in L_1[0,1]$ for~$k \in [K]$. Introduce the matrix $P = (P_{ij})_{i\in E_1,j\in E_2}$ with entries 
$$ 
P_{ij} = \int_{C_{ij}} f(x,y) dx dy = \sum_{k=1}^K \int_{A_i} u_k(x) dx \int_{B_j} v_k(y) dy = \sum_{k=1}^K U_k(i) V_k(j), \qquad  (i,j) \in E_1\!\times\! E_2, 
$$
 where  $U_k(i) = \int_{A_i} u_k(x) dx$ and $V_k(j) = \int_{B_j} v_k(y) dy$ for any $(i,j) \in E_1\!\times\! E_2$ and $k \in [K]$. Set $U_k = (U_k(i))_{i \in E_1}$, $V_k = (V_k(j))_{j \in E_2}$. Then we can write
$$
P = \sum_{k=1}^K U_k V_k^\top.
$$ 
Matrix $P$ has rank at most $K$. Consider now
the histogram matrices $G= (G_{ij})_{i \in E_1, j\in E_2}$ and $G' =(G'_{ij})_{i \in E_1, j\in E_2}$ defined in Algorithm~\ref{alg:Alg_densities} with entries 
$$
G_{ij} = \frac{4}{n}\sum_{\ell=n/2+1}^{3n/4} \mathbbm{1}_{X_\ell \in C_{ij}} \quad \text{ and } \quad G'_{ij} = \frac{4}{n}\sum_{\ell=3n/4+1}^n \mathbbm{1}_{X_\ell \in C_{ij}}, \qquad (i,j) \in E_1\!\times\! E_2.$$
The matrices $G$ and $G'$ are mutually independent, and both $nG/4$ and $nG'/4$ follow the multinomial distribution~$\mathcal{M}(P,n/4)$. 

To alleviate the notation, we define the following two quantities
\begin{align*}
    \psi_{\text{low-rank}} = (K/n)^{\frac{\beta}{2\beta+1}} \log^{3/2}(n) \quad \text{ and } \quad \psi_{2D} =  n^{-\frac{\beta}{2\beta+2}}.
\end{align*}
By the definition of Algorithm \ref{alg:Alg_densities}, matrix 
$\widehat P^*$ is the output of Alg\ref{alg:Alg}($\alpha, b, n,  \frac{n}{4}, G, G'$) with $b=|E_1| \lor |E_2|$, and $\alpha>1$ if $\psi_{\text{low-rank}} \leq \psi_{2D}$ and  $
    \widehat P^* = (G+G')/2$ if $\psi_{\text{low-rank}} > \psi_{2D}.$
Therefore, if $\psi_{\text{low-rank}} \leq \psi_{2D}$, then Theorem~\ref{prop:UB} implies that, for some constants $C>0$ depending only on $\alpha$,  
\begin{align*}%
    \|\widehat P^* - P\|_1 & \leq C \sqrt{\frac{Kb}{n}} \log(b) \log^{1/2}(n)
     \leq C\left(\frac{K}{n}\right)^{\beta/(2\beta+1)} (\log n)^{3/2}
\end{align*}
with probability at least $1- C_0(\log b)^2b n^{-\alpha}$, where $C_0>0$ depends only on $\alpha$. 

 If $\psi_{\text{low-rank}} > \psi_{2D}$, then \eqref{eq:b} implies that $b \leq Cn^{\frac{1}{2\beta+2}}$. Using this fact and Lemma~\ref{prop:samp_comp_TV_discrete} we obtain that, conditionally on $\mathcal{D}_1 = \{X_1,\dots,X_{3n/4}\}$,
\begin{align*}
    \mathbb P\bigg(\|\widehat P^* - P\|_1 > C n^{\frac{\beta}{2\beta+2}} \; \Big| \:\mathcal{D}_1\bigg)\leq \mathbb P\bigg(\|\widehat P^* - P\|_1 > \sqrt{\frac{b^2}{n/2}} + \sqrt{\frac{2\alpha\log(n)}{n/2}}\; \Big| \:\mathcal{D}_1\bigg)\leq n^{-\alpha},
\end{align*}
where the constant $C>0$ depends only on $\alpha$.

Combining the cases $\psi_{\text{low-rank}} \leq \psi_{2D}$ and $\psi_{\text{low-rank}} > \psi_{2D}$ and using the fact that $\|\widehat P^* - P\|_1\le 2$ we obtain the bound
\begin{align}\label{eq1-proof-density}
    \mathbb P\bigg(\|\widehat P^* - P\|_1 > C\big(\psi_{\text{low-rank}} \wedge \psi_{2D}\wedge 1\big)\bigg)
    \le C_0(\log b)^2b n^{-\alpha},
\end{align}
where the constant $C>0$ depends only on $\alpha$.
This bound will be used to control the stochastic component $\|\widehat \phi - \bar f\|_{L_1}$ of the $L_1$-error of the estimator $\widehat \phi$, where  
$\bar f$ is piecewise constant function defined as follows:  
$$
\bar f(x,y) = \frac{P_{ij}}{h_1 h_2}   
\qquad \text{ if } (x,y) \in C_{ij}.
$$
We have $\int_{C_{ij}} \bar f = P_{ij}$.
On the other hand, by the definition of Algorithm \ref{alg:Alg_densities},
$\widehat \phi(x,y) = \widehat P_{ij}^*/(h_1 h_2)$ for $(x,y) \in C_{i,j}$, and $\int_{C_{ij}} \widehat \phi = \widehat P_{ij}^*$.

The bias component of the error is $\|f - \bar f\|_{L_1}$. In order to control it, we need to distinguish between two cases. 
Indeed, $f$ can be discontinuous at the boundaries of its rectangular support, which requires separately analyzing the behavior of $\widehat f$ on the cells $C_{ij}$ that intersect the boundary of $\operatorname{Supp}(f)$. 
Let $i_0, i_1 \in E_1$ be the indices such that $r_1 \in A_{i_0}$ and $R_1 \in A_{i_1}$ respectively. 
Similarly, let $j_0, j_1 \in E_2$ be the indices such that $r_2 \in B_{i_0}$ and $R_2 \in B_{i_1}$ respectively. 
We note that $i_0, j_0 \leq -1$ and that $i_1 \geq \ell_1$ and $j_1 \geq \ell_2$.
We let $\mathcal{B}$ denote the indices $(i,j)$ of the cells $C_{ij}$ intersecting the boundary of $\operatorname{Supp}(f)$:
\begin{align*}
    \mathcal{B} &= \left\{(i,j) \in E_1 \times E_2 : C_{i,j} \cap \partial \operatorname{Supp}(f) \neq \emptyset\right\}\\
    & = \left\{(i,j) \in E_1 \times E_2 : 
    \begin{cases}
        i \in \{i_0,i_1\} \text{ and } j \in [j_0,j_1]\\
        \text{or}\\
        j \in \{j_0,j_1\} \text{ and } i \in [i_0,i_1]
    \end{cases}\right\}.
\end{align*}
We define  $\mathcal{C} = \bigcup_{(i,j) \in \mathcal{B}} C_{ij}$.

We first consider a cell $C_{ij}$ that does not intersect the boundary of $\operatorname{Supp}(f)$, which means that $C_{ij}\in \mathcal{C}^c$, that is, $(i,j)\in (E_1 \times E_2) \setminus \mathcal{B}$.  
For $(x,y)$ in such cells $C_{ij}$ we have 
$$
\left|f(x,y) - \bar f(x,y)\right| = \left|\frac{1}{h_1 h_2}\int_{C_{ij}} (f(x,y) - f(x',y'))dx'dy'\right| \leq C L (h^*)^{\beta},
$$
which yields that 
\begin{equation}\label{eq2-proof-density}
\|f - \bar f\|_{L_1(\mathcal{C}^c)}=\sum_{(i,j)\in (E_1 \times E_2) \setminus \mathcal{B}} \int_{C_{ij}} \left|f(x,y)- \bar f(x,y) \right|dxdy 
 \le CL(h^*)^{\beta}. 
\end{equation}
We consider now the opposite case $(i,j) \in \mathcal{B}$, and we analyze the behavior of $\widehat \phi$ on $\mathcal{C}$. Note that, by construction, the sets  $C_{ij}$ with $(i,j) \in \mathcal{B}$ cannot belong to the rectangle $[\widehat r_1,\widehat R_1] \times [\widehat r_2,\widehat R_2]$, which is included in the interior of $\operatorname{Supp}(f)$ and represents a union of sets $C_{ij}$ with some $(i,j)$'s.   Therefore, 
on the event $\mathcal{E}$, we have $\int_{\mathcal{C}} f \leq C {n^{-1/3}\leq n^{-\beta/(2\beta+1)}}$ for an absolute constant $C>0$. It follows that, on the event $\mathcal{E}$,
\begin{equation}\label{eq3-proof-density}
     \|f - \Bar{f}\|_{L_1(\mathcal{C})}\le 
\|f \|_{L_1(\mathcal{C})} +\|\Bar{f} \|_{L_1(\mathcal{C})}=2 \|f \|_{L_1(\mathcal{C})} \le Cn^{-{\beta}/(2\beta+1)}.
\end{equation}
Combining \eqref{eq1-proof-density}--\eqref{eq3-proof-density} we obtain that 
\begin{align}
    \|f - \widehat \phi\|_{L_1}
    &\leq \|f - \Bar{f}\|_{L_1} + \|\Bar{f} - \widehat \phi\|_{L_1}
    \nonumber\\[5pt]
    &\le \|f - \Bar{f} \|_{L_1(\mathcal{C}^c)} + \|f - \Bar{f} \|_{L_1(\mathcal{C})}
    +\sum_{(i,j)\in E_1 \times E_2} \int_{C_{ij}} |\Bar{f} - \widehat \phi|
    \nonumber\\
    & \leq CL(h^*)^{\beta} + Cn^{-\beta/(2\beta+1)}
     + \sum_{(i,j) \in E_1 \times E_2}  |\widehat P_{ij}^* - P_{ij}|
     \nonumber\\
    & \le  C\left(\left(\frac{K}{n}\right)^{\beta/(2\beta+1)} \log^{3/2}\! n \land n^{-\frac{\beta}{2\beta+2}}\right)
    \label{eq4-proof-density}
\end{align}
with probability at least 
$1- C_0(\log b)^2  b n^{-\alpha}-{4\exp(-n^{2/3}/2)}$ 
Note also that $b\le Cn^{1/(2\beta+1)}$. This implies that there exist constants $C_0'>0, C>0$ depending only on $\alpha$ and $L$ such that the bound \eqref{eq4-proof-density} holds with probability at least 
$1- C_0' (\log n)^2 n^{1/(2\beta+1)-\alpha}$. Next, since $\|\widehat \phi\|_{L_1}\le 1$ we have $\|f - \widehat \phi\|_{L_1}\le 2$. Thus, we can replace the bound in \eqref{eq4-proof-density} by a stronger bound  
$C\left\{\left(\frac{K}{n}\right)^{\beta/(2\beta+1)} \log^{3/2}\! n \land n^{-\frac{\beta}{2\beta+2}}\right\}$
that holds with the same probability. 
Furthermore, a bound of the same order is satisfied for $\|f - \widehat f\|_{L_1}$ with the same probability. Indeed, an argument analogous to the proof of \eqref{reduction to probability} yields that $\|f - \widehat f\|_{L_1}\le 2\|f - \widehat \phi\|_{L_1}$. Thus, the bound \eqref{eq1:th:upper bound-density estimation} of Theorem~\ref{th:upper bound-density estimation} is proved. Next, if $\alpha> 4/3$ then the bound \eqref{eq2:th:upper bound-density estimation} for the expectation follows easily 
from~\eqref{eq1:th:upper bound-density estimation} and the inequality $\|f - \widehat f\|_{L_1}\le 2$. Finally, note that the condition $n\ge 64$ used to apply Lemma~\ref{lem:density_1D} can be dropped since for $n<64$ the result of the theorem follows from the trivial bound $\|f - \widehat f\|_{L_1}\le 2$.
\end{proof}

\begin{lemma}\label{lem:density_1D}
Let $f: \R\longrightarrow \R_+$ be a probability density. Assume that for some $r,R \in [0,1]$, the function $f$ is $\beta$-H\"older on $[r,R]$ and that $f=0$ on $\R\setminus [r,R]$. 
Let  $Z_1,\dots, Z_n$ be iid random variables distributed with probability density $f$, where $n\ge 64$ is an even integer.
{Define $\widehat r = \min \left\{Z_i: i \in \{1,\dots,  \frac{n}{2}\}\right\}$ and  
$\widehat R = \max \left\{Z_i: i \in \{1,\dots,  \frac{n}{2} \}\right\}$.}
Let also $\Delta = R-r$ and $\widehat \Delta = \widehat R - \widehat r$.
Let $q_-$ and $q_+$ be the quantiles of order $n^{-1/3}$ and $1-n^{-1/3}$ of the probability measure induced by $f$:
    $$ \int_{-\infty}^{q_-}f(x)dx = n^{-1/3}, \quad  \quad \int_{q_+}^{+\infty}f(x)dx = n^{-1/3}. 
    $$
    Define the event $\mathcal{E} = \left\{\widehat r < q_- < q_+ < \widehat R\right\}$ and set $h^* = (K'/n)^{1/(2\beta+1)}$, where $K'\ge 1$ is such that $K'/n \le (8L)^{-3/2}\wedge 1$.
Let $\widehat g$ be an output of {\rm Alg\ref{alg:Alg_densities_1D}}$(Z_1,\dots, Z_n,K')$. Then the following holds.
\begin{enumerate}
    \item $\mathbb P(\mathcal{E}) \geq 1 - 2e^{-n^{2/3}/2}$.
    \item\label{item:Delta_leq_hat_Delta} On the event $\mathcal{E}$, if $\widehat \Delta \leq h^*$ then $\Delta < \widehat \Delta \left[1 + 16 n^{-1/(2\beta+1)}\right]$.
    \item There exists a constant $C>0$ depending only on $L$ such that 
$$\mathbb P\left(\|f - \widehat g\|_{L_1} \leq C (K'/n)^{\beta/(2\beta+1)}\right) \geq 1 - {4}\exp\big(\!-n^{1/3}\big).
$$ 
\end{enumerate}
\end{lemma}

\vspace{1mm}

\begin{proof}[Proof of Lemma~\ref{lem:density_1D}]
    The first assertion of the lemma follows from the inequalities 
    \begin{align}\label{eq:control_E}
        \mathbb P(\mathcal{E}^c) \leq {\mathbb P(\widehat r \ge q_- ) +} \mathbb P(\widehat R {\le } q_+ ) = {\mathbb P\left(Z_1\ge q_-\right)^{\frac{n}{2}} + \mathbb P\left(Z_1 \le q_+\right)^{\frac{n}{2}} = 2\left(1-n^{-1/3}\right)^{\frac{n}{2}} \leq 2e^{-n^{2/3}/2}.}
    \end{align}
   We now prove the second assertion of the lemma. On the event $\mathcal{E}$ and 
if $\widehat \Delta < h^*$, using the {$\beta$-H\"older} continuity of $f$ we have
    \begin{align*}
        \frac{1}{2} &\leq 1 - 2n^{-1/3} \leq \int_{\widehat r}^{\widehat R} f(x)dx \leq \widehat \Delta \max_{[\widehat r, \widehat R]} f \\
        &\leq \widehat \Delta \min_{[\widehat r, \widehat R]} f + L \widehat \Delta^{\beta+1} \leq \widehat \Delta \min_{[\widehat r, \widehat R]} f + L(h^*)^{\beta+1} \\
        &\leq \widehat \Delta \min_{[\widehat r, \widehat R]} f + \frac{1}{8},
    \end{align*}
    where the last inequality follows from the fact that $L(h^*)^{\beta+1} = L(K'/n)^{(\beta+1)/(2\beta+1)}  \leq \frac{1}{8}$ due to the assumption $K'/n \le (8L)^{-3/2}\wedge 1$.
    Therefore, $f(\widehat r) \geq \min_{[\widehat r, \widehat R]} f \geq \frac{3}{8\widehat \Delta}$.
    
    Set $\overline r := \widehat r - \frac{2}{f(\widehat r)}n^{-1/(2\beta+1)}$ and let us prove that $r \ge  \overline r$. Indeed, assume that, on the contrary, $r < \overline r$. 
    Then, on the event $\mathcal{E}$ and 
if $\widehat \Delta \le h^*$, using the {$\beta$-H\"older} continuity of $f$ we get
    \begin{align*}
        n^{-1/3} &\geq \int_{r}^{\widehat r} f(x)dx \geq \int_{\overline r}^{\widehat r} f(x)dx \geq (\widehat r-\overline r) f(\widehat r) - L(\widehat r - \overline r)^{\beta+1}/(\beta+1) 
        \\
        & = 2 n^{-1/(2\beta+1)} - \frac{L}{\beta+1}\Big(\frac{2}{f(\widehat r)}\Big)^{\beta+1} n^{-(\beta+1)/(2\beta+1)} 
        \\ 
        & \geq 
        2 n^{-1/3} - \frac{L}{\beta+1}\Big(\frac{16 \widehat \Delta}{3}\Big)^{\beta+1} n^{-(\beta+1)/(2\beta+1)}
        \\
        & \geq 
        2 n^{-1/3} - \frac{L(h^*)^{\beta+1}}{\beta+1}\Big(\frac{16 }{3}\Big)^{\beta+1} n^{-(\beta+1)/(2\beta+1)}
        \\
        & \geq 2 n^{-1/3} - \frac{1}{8(\beta+1)}\Big(\frac{16 }{3}\Big)^{\beta+1} n^{-2/3} > n^{-1/3} 
    \end{align*}
    for all $n \ge 64$ and $\beta\in(0,1]$, which is a contradiction. 
    Thus, we have $r\ge \overline r \ge \widehat r - 8\widehat \Delta n^{-1/(2\beta+1)}$ and, similarly, $R \le \widehat R + 8\widehat \Delta n^{-1/(2\beta+1)}$, which implies the desired inequality $\Delta \leq \widehat \Delta (1+ 16 n^{-1/(2\beta+1)})$. This concludes the proof of the second assertion of the lemma.
    \vspace{1mm}

    Finally, we prove the third assertion of the lemma. We have, for any $t>0$,
    \begin{align}\label{eq:third-assertion}
        \mathbb P(\|f - \widehat g\|_{L_1} > t) &\le \mathbb P\left(\{\|f - \widehat g\|_{L_1} > t\}\cap \mathcal{E} \cap \{\widehat \Delta \le {h^*}\}\right)
        \\ \nonumber
        & \quad + 
        \mathbb P\left(\{\|f - \widehat g\|_{L_1} > t\}\cap \mathcal{E}\cap \{\widehat \Delta > {h^*}\}\right) + \mathbb P(\mathcal{E}^c).
    \end{align}
Consider the first probability on the right hand side of \eqref{eq:third-assertion}.
    Recall that if $\widehat \Delta < {h^*}$
    the estimator of Algorithm~\ref{alg:Alg_densities_1D} is $\widehat g = \frac{1}{\widehat R - \widehat r} \mathbbm{1}_{[\widehat r, \widehat R]}$. 
By continuity of $f$ over $[r,R]$, there exists $x_0 \in [r,R]$ such that $f(x_0)(R-r) = \int_r^R f = 1$.
    Thus, using the second assertion of the lemma we obtain that, on the event $\mathcal{E}$ and if $\widehat \Delta \le {h^*}$,
    \begin{align*}
    \|f-\widehat g\|_{L_1} &= \int_r^R |f(x)-\widehat g(x)|dx = \int_{r}^{\widehat r} f(x)dx + \int_{\widehat r}^{\widehat R} |f(x)-\widehat g(x)|dx + \int_{\widehat R}^R f(x)dx 
    \\
    & \leq \int_{\widehat r}^{\widehat R} |f(x)-f(x_0)|dx + \int_{\widehat r}^{\widehat R} |f(x_0)-\widehat g(x)|dx + 2n^{-1/3}
    \\
    & \leq L|R-r|^{\beta+1} + \int_{\widehat r}^{\widehat R} |f(x_0)-\widehat g(x)|dx +  2n^{-1/3}
    \\
    & \leq L\left(\widehat \Delta + 16 \widehat \Delta  n^{-1/(2\beta+1)}\right)^{\beta+1} + (\widehat R-\widehat r) \left|\frac{1}{R-r} - \frac{1}{\widehat R - \widehat r}\right| +  2n^{-1/3}
    \\
    & \leq 25L(h^*)^{\beta+1} +{\frac{\Delta - \widehat \Delta }{\Delta}} 
    +  2n^{-1/3} 
   \\ 
    & {\leq 25L(h^*)^{\beta+1} +\frac{\widehat \Delta  16 n^{-1/(2\beta+1)}}{\Delta}
    +  2n^{-1/3} }\\
    & \leq 25L(K'/n)^{(\beta+1)/(2\beta+1)} + 18 n^{-1/3} \quad (\text{since} \quad {\widehat \Delta }/{\Delta} \le 1) \\
    & \leq (25L +18) (K'/n)^{\beta/(2\beta+1)} 
\end{align*}
for all $n\ge 64$ and $\beta\in (0,1]$. We conclude that the first probability on the right hand side of \eqref{eq:third-assertion} vanishes for all $t> (25L +18)(K'/n)^{\beta/(2\beta+1)}$.

Next, consider the second probability on the right hand side of \eqref{eq:third-assertion}. Fix the subsample $\mathcal{D}_1 = (Z_1,\dots, Z_{ n/2})$ such that
 $\widehat \Delta > {h^*}$. 
Then the partition $(A_j)_{j\in E}$ defined in Algorithm~\ref{alg:Alg_densities_1D} is also fixed, where $E=\{1,\dots,\ell\}$.
By continuity of $f$ over $A_j$, we define $x_j \in A_j$ such that $f(x_j) \, h = \int_{A_j} f$. 
For any $j \in E$, we define 
\begin{align*}
    p_j = \int_{A_j} f(x)dx, \quad \text{ and } \quad \widehat p_j = \int_{A_j} \widehat g(x)dx \;=\, \frac{2N_j}{n},
\end{align*}
where $N_j=\sum_{i=\lfloor n/2\rfloor+1}^n \mathbbm{1}_{Z_i \in A_j}$, see the definition of Algorithm~\ref{alg:Alg_densities_1D}.
Note that $\widehat g$ is supported on $[\widehat r,\widehat R]$ and $(A_j)_{j\in E}$ is a partition of $[\widehat r,\widehat R]$. Then we have the following bound on the estimation error of $\widehat g$, which is valid for any fixed subsample $\mathcal{D}_1$ such that {event $\mathcal E$ holds and}
 $\widehat \Delta > {h^*}$:
\begin{align} \nonumber
    \|f - \widehat g\|_{L_1} &=  {\int_r^{\widehat r}f(x)dx + \int_{\widehat R}^Rf(x)dx} + \sum_{j\in E} \, \int_{A_j} \big|f(x)-\widehat g(x)\big|dx 
    \\ \nonumber
    &\le 2n^{-1/3} + \sum_{j\in E} \, \int_{A_j} \big|f(x)-\widehat g(x)\big|dx  \quad \text{(since $\mathcal E$ holds)}
    \\ \label{eq:third-assertion-1}
    &\leq 2n^{-1/3} +  \sum_{j\in E} \, \int_{A_j} \big(\big|f(x) - f(x_j)\big| 
    + \big|f(x_j)-\widehat g(x)\,\big|\big) \, dx
    \\
    & \leq 2n^{-1/3} + \sum_{j\in E} \left\{ Lh^{\beta+1} + |p_j - \widehat p_j|\right\} \nonumber
    \\
    & \le 2n^{-1/3} +  Lh^{\beta} + \sum_{j=1}^\ell |p_j - \widehat p_j| \quad \text{(since $h={\widehat \Delta}/{\ell}\le 1/\ell$)} \nonumber
        \\
    & \le 2n^{-1/3} +  {2}L(K'/n)^{\beta/(2\beta+1)} + \sum_{j=1}^\ell |p_j - \widehat p_j|,  \nonumber
\end{align}
where the last inequality uses the fact that $h=\frac{\widehat \Delta}{\lfloor\widehat \Delta/h^*\rfloor}=h^*\frac{\widehat \Delta/h^*}{\lfloor\widehat \Delta/h^*\rfloor}\le 2h^*= 2(K'/n)^{1/(2\beta+1)}$ for $\widehat \Delta>h^*$.
Lemma~\ref{prop:samp_comp_TV_discrete} stated below yields that, for any $\delta \in (0,1)$, 
\begin{align*}
    &\mathbb P\left(\sum_{j=1}^\ell  |p_j - \widehat p_j| > 2\sqrt{\frac{2(\ell + \log(1/\delta))}{n}}\;\Big|\; \mathcal{D}_1 \right) \le \delta.
\end{align*}
Set $\delta = \exp\big(\!-n^{1/(2\beta+1)}\big)$. Then combining this bound with \eqref{eq:third-assertion-1} 
and the fact that $\ell = \big\lfloor \widehat \Delta /h^*\big\rfloor\le 1/h^*\le n^{1/(2\beta+1)}$ we obtain that there exists a constant $C_*>0$ depending only on $L$ such that
\begin{align*}
    &  \mathbb P\left(\|f - \widehat g\|_{L_1} > C_*
(K'/n)^{\beta/(2\beta+1)} \;\Big|\; \mathcal{D}_1 \right) \le \exp\big(\!-n^{1/(2\beta+1)}\big)
    \quad \text{if $\mathcal{D}_1$ is such that {$\mathcal E$ holds} and 
 $\widehat \Delta > {h^*}$. }
\end{align*}
It follows that 
$$
\mathbb P\left(\{\|f - \widehat g\|_{L_1} > C_* (K'/n)^{\beta/(2\beta+1)}\}\cap \mathcal{E}\cap \{\widehat \Delta > {h^*}\}\right)\le \exp\big(\!-n^{1/(2\beta+1)}\big).
$$
Plugging this bound and \eqref{eq:control_E} in \eqref{eq:third-assertion} and recalling that the first probability on the right hand side of \eqref{eq:third-assertion} vanishes for $t> (25L +18) (K'/n)^{\beta/(2\beta+1)}$  we find that, for $C>\max(25L +18, C_*)$,
$$
\mathbb P (\|f - \widehat g\|_{L_1} > C (K'/n)^{\beta/(2\beta+1)})\le \exp\big(\!-n^{1/(2\beta+1)}\big) + \mathbb P (\mathcal E^c)\le {4}\exp\big(\!-n^{1/3}\big),
$$
where we have used the first assertion of the lemma and the fact that {$\exp\big(\!-n^{2/3}/2\big)\le \exp\big(\!-n^{1/3}\big)$ for $n\geq 8$. }
\end{proof}

\subsection{Lower bounds for continuous distributions}

\begin{proof}[Proof of Theorem~\ref{th:lower bound for density estimation}] 
We define $f_0 = \mathbbm{1}_{[0,1]^2}$ and set $K' = \lfloor n^{1/(2\beta+2)} \rfloor \land K$, $H=\lceil (n/K')^{1/(2\beta+1)} \rceil$, $h_x=1/K'$, $h_y=1/H$.
Let $\varphi:\mathbb{R}\to [0,1]$ be a non identically zero infinitely many times differentiable function with support $(-1/2,1/2)$ satisfying the $\beta$-H\"older condition with H\"older constant $1/2$ (see (2.34) in \cite{tsybakov2009introduction} for an example of such function).
For $i,j \in \mathbb Z$, set $U_i(x) := \mathbbm 1_{[x_i \pm h_x/2]}(x)$ where $x_i = \big(i-\frac{1}{2}\big) h_x$ 
and $V_{j,h_y}(y) = \varphi\big(\frac{y-y_j^-}{h_y/2}\big)-\varphi\big(\frac{y-y_j^+}{h_y/2}\big)$, where
$y_j^- = \big(j-\frac{3}{4}\big)h_y$ and $y_j^+ = \big(j-\frac{1}{4}\big)h_y$. 
The supports of the functions $\varphi\big(\frac{\cdot-y_j^-}{h_y/2}\big)$ and $\varphi\big(\frac{\cdot-y_j^+}{h_y/2}\big)$ are disjoint and each of these functions is $\beta$-H\"older with H\"older constant $(1/2)(h_x/2)^{-\beta}$. Hence, $V_{j,h_y}(\cdot)$ is $\beta$-H\"older with H\"older constant $(h_y/2)^{-\beta}$. Note also that the functions $V_{j,h_y}$  
have disjoint supports in $[0,1]$ for different $j$'s and that the functions $V_{j,h_y}$ integrate to 0.  
Similarly, the functions $U_i$ have disjoint supports in $[0,1]$ and are $\beta$-Hölder on their support with Hölder constant $0$ since they are constant over their support.

Set $c_*=1/4\wedge 1/(2L)$.  
For $\omega = (\omega_{ij})$, where $\omega_{ij}\in \{0,1\}$ for all $i,j$, we define the functions $f_\omega$ as follows:
\begin{align}
    f_\omega(x,y) &= \sum_{i=1}^{K'} U_i(x) \bigg(1 + \sum_{j=1}^{H}  \omega_{ij} c_* L h_y^\beta V_{j,h_y}(y)\bigg).\label{eq_def_prior}
\end{align}
The functions $f_\omega$ belong to the class $\mathcal{F}_{K,\beta}^{\text{mixtures}}$ since they can be rewritten as mixtures of separable densities that are products of Hölder-smooth $1$-dimensional densities. 
Indeed, it holds that
\begin{align*}
    f_\omega(x,y) = \frac{1}{K'} \sum_{k=1}^{K'} K' U_i(x)\bigg(1 + \sum_{j=1}^{H}  \omega_{ij} c_* L h_y^\beta V_{j,h_y}(y)\bigg).
\end{align*}
The functions $K' U_i$ are densities over $[0,1]$ and are $\beta$-Hölder over their support since $U_i$ is a constant function over an interval. 
Moreover, it suffices to take $c_*$ small enough to guarantee that $1 + \frac{1}{C_0}\sum_{j=1}^{H}  \omega_{ij} c_* L h_y^\beta V_{j,h_y}$ is non-negative, ensuring that $1 + \frac{1}{C_0}\sum_{j=1}^{H}  \omega_{ij} c_* L h_y^\beta V_{j,h_y}$ is a $\beta$-Hölder density over $[0,1]$ since the $V_{j,h_y}$'s have a zero integral and are $\beta$-Hölder.

First, we check that each $f_\omega$ belongs to $\mathcal{G}_{K',{\beta}} \subseteq \mathcal{G}_{K,{\beta}}$.  
\begin{itemize}
    \item We have $\int_{[0,1]^2} f_\omega = 1$ by construction.
    \vspace{-2mm}
    \item 
   We have $f_\omega(x,y) \geq {1}/{2}$ for all $(x,y)\in [0,1]^2$ since $c_*\le (2L)^{-1}$, $h_y\le 1$, and the functions $(x,y) \mapsto U_i(x) V_{j,h_y}(y)$ have disjoint supports and take values in $[-1,1]$.
    \vspace{-2mm}
    \item We now check the H\"older condition. 
    Let $C_{ij} = [x_i \pm h_x/2] \times [y_j \pm h_y/2]$ where $y_j = (j-1/2)h_y$ for any $j\in\mathbb Z$.
   Note that for each $i \in [K']$ and $j \in [H]$, the cell $C_{ij}$  contains the support of the function $(x,y) \mapsto U_i(x) V_{j,h_y}(y)$. 
   Since the cells $C_{ij}$ are disjoint and $f_\omega = 1$ at the boundary of each cell $C_{ij}$, it suffices to show that $f_\omega$ is $\beta$-H\"older with H\"older constant $L/2$ on each cell $C_{ij}$ to obtain that $f_\omega$ is $\beta$-H\"older with H\"older constant $L$ on $[0,1]^2$. Fix any two points $z,z' \in [0,1]^2$, belonging to the same cell $C_{ij}$. Then, writing $z = (x,y)$, $z' = (x',y')$ and recalling that the function $\varphi$ is $\beta$-H\"older with H\"older constant $1/2$ we obtain: 
    \begin{align*}
        \big|f_\omega(z) - f_\omega(z')\big| &= c_*Lh_y^\beta\,\Big|U_i(x) V_{j,h_y}(y) - U_i(x') V_{j,h_y}(y')\Big|\\
        &= c_*Lh_y^\beta \bigg|\varphi\Big(\frac{y - y_j^+}{h_y/2}\Big) - \varphi\Big(\frac{y'- y_j^+}{h_y/2}\Big) - \Big[\varphi\Big(\frac{y - y_j^-}{h_y/2}\Big) + \varphi\Big(\frac{y' - y_j^-}{h_y/2}\Big)\Big]\bigg| 
         \\
        & \leq c_* Lh_y^\beta \left(\frac{|y-y'|}{h_y/2}\right)^\beta 
          \\
        &\leq \frac{L}{2} \|z-z'\|^\beta_{\infty},
    \end{align*}
    where we have used the facts that $h_y\leq h_x$ and $c_*\le 1/4$.
    \vspace{-2mm}
    \item Function $f_\omega$ belongs to $\mathcal{F}_{K'} \subseteq \mathcal{F}_{K}$ since it admits a separable representation given in~\eqref{eq_def_prior}.
\end{itemize}

Thus, $\mathcal{G}':=\{f_\omega: \omega\in \{0,1\}^{K'\times H}\}$ is a subset of $\mathcal{G}_{K',{\beta}}$, and it suffices to prove the lower bound with the required rate on this subset. We first prove the lower bound in expectation~\eqref{eq:th:LBdiscr3}. We use the version of Assouad's lemma
given in \cite[Theorem 2.12(\it{iv})]{tsybakov2009introduction}.
For any $\omega, \omega' \in \{0,1\}^{K' \times H}$ that only differ in one entry, that is, for exactly one $(i_0,j_0)$, we have $\omega_{i_0,j_0}\neq \omega'_{i_0,j_0}$ and $\omega_{i,j}= \omega'_{i,j}$ for all other $i,j$, the $\chi^2$-divergence between the densities $f_\omega$ and $f_{\omega'}$ is bounded as follows:
\begin{align*}    \chi^2\left(f_{\omega'}, f_{\omega}\right) &= \int \frac{\big(U_{i_0}(x)c_* L h_y^\beta V_{j_0,h_y}(y)\big)^2}{f_\omega(x,y)} \leq \int_{C_{i_0,j_0}} 2(c_*Lh_y^\beta)^2 = 2(c_*Lh_y^\beta)^2 h_x h_y \leq \frac{ \widetilde{c}}{n},
\end{align*}
where $\widetilde{c}>0$ depends only on $L, \beta$. Using  Lemma 2.7 in ~\cite{tsybakov2009introduction} we obtain that  the $\chi^2$-divergence between the corresponding product densities satisfies
\begin{align*}
    \chi^2\left(f_{\omega'}^{\otimes n}, f_{\omega}^{\otimes n}\right) = \left(1+\chi^2\left(f_{\omega'}, f_{\omega}\right)\right)^n-1 \leq e^{\widetilde{c}}-1.
\end{align*}
Moreover, the functions $f_\omega$ and $ f_{\omega'}$ are separated in $L^1$ norm as follows
\begin{align*}
     \|f_\omega - f_{\omega'}\|_{L^1} = \big\|(x,y) \mapsto 2 c_* L h_y^{\beta}\, U_{i_0} V_{j_0,h_y} \big\|_{L_1}
     & 
     = 2 c_* L h_y^\beta h_x \int \varphi\Big(\frac{y}{h_y/2}\Big)  dy
     = c_* L h_x h_y^{\beta+1} \int \varphi(t)  dt.
\end{align*}
Therefore, applying \cite[Theorem 2.12(\it{iv})]{tsybakov2009introduction} and arguing as in  \cite[Example 2.2]{tsybakov2009introduction} we obtain 
\begin{align}\label{eq-lb-in-exp}
 \inf_{\widehat f} \sup_{f_\omega\in \mathcal{G}'} \mathbb E  
    \|\widehat f-f_\omega\|_{L_1}
     &\geq \frac{K'H}{2} (c_*Lh_y^\beta) c'' h_x h_y  = \frac{c_*c''}{2} L h_y^\beta  \geq c \Big((K/n)^{\beta/(2\beta+1)} \land n^{-{\beta}/{(2\beta+2)}}\Big),   
\end{align} 
    where $c'' = \exp(1-e^{\widetilde{c}}) \int \varphi(x) dx $ and  $c>0$ is a constant depending only on $L$ and $\beta$.    This implies the lower bound in expectation~\eqref{eq:th:LBdiscr3}. 
    \vspace{1mm}
    
    To obtain the lower bound in probability~\eqref{eq1:th:lower bound for density estimation} we apply Lemma~\ref{lem:reduction of lower bound}
    with $\mathcal{P} = \mathcal{P}_0 = \mathcal{G}'$,  $v(f,g) = \|f-g\|_{L^1}$ for any $ f,g \in \mathcal{G}'$, and 
    $U = \mathbbm 1_{[0,1]^2} \in \mathcal{G}'$.  
    Note that there exists a constant $C>0$ such that for any $f \in \mathcal{G}'$ we have
    $
        v(U,f) \leq s,
    $ 
    where $s =  C \big(\left(K/n\right)^{\beta/(2\beta+1)} \land n^{-{\beta}/{(2\beta+2)}}\big)$. 
    Therefore, taking into account the bound in expectation \eqref{eq-lb-in-exp} we can apply Lemma~\ref{lem:reduction of lower bound} with  a small enough constant $a>0$ to get~\eqref{eq1:th:lower bound for density estimation}.
\end{proof}

\subsection{Upper bound for the adaptive estimator}\label{sec:proof:adaptive}

\begin{lemma}\label{lem:adaptive}
 Let $0<\delta<1$. For any probability density $f$ we have that, with $\mathbb{P}_f$-probability at least $1-\delta$, 
 \begin{equation}
 \label{eq1:lem:adaptive} 
  \|\widehat f^* -f\|_{L_1} \le 3 \min_{(K,j)} \|\widehat f_{(K,j)} -f\|_{L_1} + 2\sqrt{\frac{2\log(2m^2/\delta)}{n}},
 \end{equation}
and 
 \begin{equation}
 \label{eq2:lem:adaptive}   
  \mathbb{E}_f\|\widehat f^* -f\|_{L_1} \le 3 \min_{(K,j)} \mathbb{E}_f \|\widehat f_{(K,j)} -f\|_{L_1} + 2\sqrt{\frac{2\log(2m^2)}{n}}.
 \end{equation}
\end{lemma}
\begin{proof}
    This lemma is a simple corollary of Theorem 6.3 in \cite{DevroyeLugosi}, which states that 
    $$
  \|\widehat f^* -f\|_{L_1} \le 3 \min_{(K,j)} \|\widehat f_{(K,j)} -f\|_{L_1} + 4\Delta,
 $$
 where $\Delta=\underset{B\in \mathcal{B}}{\max}\Big| \int_B f - \mathbb{P}_n(B) \Big|$.
Since $|\mathcal{B} |=m(m-1)$ it follows from Hoeffding's inequality and the union bound that, for any $t>0$,
$$
\mathbb{P}_f(\Delta >t)\le 2 m(m-1)\exp(-2nt^2).
$$
This proves \eqref{eq1:lem:adaptive}. Inequality \eqref{eq2:lem:adaptive} follows from the fact that  
$\mathbb{E}_f\, \underset{B\in \mathcal{B}}{\max}\Big| \int_B f - \mathbb{P}_n(B) \Big|\le \sqrt{\frac{\log( 2|\mathcal{B} |)}{2n}}$ (see \cite[Lemma 2.2]{DevroyeLugosi}).
\end{proof}

\begin{proof}[Proof of Theorem~\ref{th:adaptive_rate}]
    To establish the adaptivity of $\widehat f^*$, we will use the inclusions between the Hölder classes:
    \begin{align*}
        \forall \; 0< \beta \le \beta' \leq 1: \quad  \mathcal{L}_{\beta'} 
        \subseteq \mathcal{L}_\beta.
    \end{align*} 
    Indeed, for any $0< \beta \le \beta' \leq 1$, any $f \in \mathcal{L}_{\beta'}$ and any $z,z' \in \Supp(f)$, we have $\|z-z'\|_\infty \leq 1$, and consequently
    \begin{align*}
        |f(z)-f(z')| \le L\|z-z'\|_\infty^{\beta'} \leq L\|z-z'\|_\infty^{\beta}.
    \end{align*}
    The embedding of the classes $(\mathcal{G}_{K,\beta})_{\beta \in (0,1]}$ immediately follows:
    \begin{align*}
        \forall \; 0< \beta \le \beta' \leq 1, ~\forall K \in \N: \quad  \mathcal{G}_{K,\beta'} 
        \subseteq \mathcal{G}_{K,\beta}.
    \end{align*} 
    Assume first that $\beta \in [\beta_j, \beta_{j-1}]$ for some $j \in \{2,\dots, \lceil \log(n) \log\log(n)\rceil \}$. 
    Using the inclusion $\mathcal{G}_{K,\beta} \subseteq \mathcal{G}_{K,\beta_j}$ and~\eqref{eq2:lem:adaptive} we can bound from above the risk of $\widehat f^*$ over the class $\mathcal{G}_{K,\beta}$ as follows:
    \begin{align}
        \sup_{f \in \mathcal{G}_{K,\beta}}\mathbb{E}_f\|\widehat f^* -f\|_{L_1} &\le \sup_{f \in \mathcal{G}_{K,\beta}}  \left\{ 3\min_{(K,i)} \mathbb{E}_f \|\widehat f_{(K,i)} -f\|_{L_1} + C\sqrt{\frac{\log(m)}{n}}\right\}
        \nonumber\\
        & \leq  \sup_{f \in \mathcal{G}_{K,\beta}} 3\; \mathbb{E}_f \|\widehat f_{(K,j)} -f\|_{L_1}+ C\sqrt{\frac{\log\big(K_{\max} \log(n) \log \log(n)\big)}{n}}
        \nonumber\\
        & \leq \sup_{f \in \mathcal{G}_{K,\beta_j}} 3\; \mathbb{E}_f \|\widehat f_{(K,j)} -f\|_{L_1} + C \sqrt{\frac{\log(n)}{n}} \hspace{14mm} \text{ since } K_{\max} = \lceil \sqrt{n}\rceil
        \nonumber\\
        & \leq C\left\{\left(\frac{K}{n}\right)^{\beta_j/(2\beta_j+1)} \log^{3/2}\! n \land n^{-\frac{\beta_j}{2\beta_j+2}}\right\}
        \qquad \text{ by Theorem~\ref{th:upper bound-density estimation}}. \label{eq_intermeriary_risk_f_star}
    \end{align}
     Note that, for $c \in \{1,2\}$,
    \begin{align*}
         \frac{\beta}{2\beta+c} - \frac{\beta_j}{2\beta_j+c} \leq \frac{c(\beta_{j-1} - \beta_j)}{(2\beta+c)(2\beta_j+c)} = \frac{c\beta_j}{(2\beta+c)(2\beta_j+c) \log(n)} \leq \frac{1}{\log(n)}.
    \end{align*}
    Combining this with the bound~\eqref{eq_intermeriary_risk_f_star} we obtain
    \begin{align*}
        \sup_{f \in \mathcal{G}_{K,\beta}}\mathbb{E}\|\widehat f^* -f\|_{L_1} & \leq C \Big\{ (n/K)^{-\beta/(2\beta+1)+ \frac{1}{\log(n)}} (\log n)^{3/2} \land n^{-\frac{\beta}{2\beta+2} + \frac{1}{\log(n)}}  \Big\} \\
        & \leq C \Big\{ (n/K)^{-\beta/(2\beta+1)} (\log n)^{3/2} \land n^{-\frac{\beta}{2\beta+2}}  \Big\}.
    \end{align*}
    This proves the theorem for $\beta\in [\beta_M, \beta_1]= [\beta_M, 1]$, where $M = \lceil\log(n) \log \log(n)\rceil$. Finally, consider the values $\beta \in (0, \beta_M)$.  
    We have
    \begin{align*}
        \frac{\beta}{2\beta+2} \leq \frac{\beta}{2 \beta + 1} \leq \frac{\beta_M}{2 \beta_M + 1} \leq \beta_M = \left(1+\frac{1}{\log(n)}\right)^{-M} \leq \exp\left(-\frac{\log(n) \log\log(n)}{\log(n)}\right) = \frac{1}{\log(n)}.
    \end{align*}
 Thus, for $\beta \in (0, \beta_M)$, 
    \begin{align*}
        (K/n)^{\beta/(2\beta+1)} \land n^{-\beta/(2\beta+2)} 
         \geq n^{-\frac{1}{\log(n)}} 
        = 1/e.
    \end{align*}
    The desired result follows immediately from this inequality and the fact that $\sup_{f \in \mathcal{G}_{K,\beta}}\mathbb{E}_f\|\widehat f^* -f\|_{L_1} \leq 2$.

\end{proof}

\section{Auxiliary results}\label{sec:auxiliary}

First, we recall the definition of the multinomial distribution. 
Given a finite set $\Omega$ and  a probability distribution $P = (P_\omega)_{\omega \in \Omega}$ on $\Omega$, we say that a random vector $Y=(Y_\omega)_{\omega \in \Omega}$ follows the multinomial distribution $\M(P,n)$, $n \in \N^*$,  if for all $(n_\omega)_{\omega \in \Omega} \in \N^{\Omega}$ satisfying $\sum\limits_{\omega \in \Omega} n_\omega = n$, we have:
$$ 
\mathbb{P}\big( Y_\omega = n_\omega, \forall \omega \in \Omega \big)= 
\frac{n!}{\prod_{\omega \in \Omega} n_\omega!}
\prod_{\omega \in \Omega} P_\omega^{n_\omega}.
$$
We denote by $\Poi(\lambda)$ the Poisson distribution with mean $\lambda$.

\begin{lemma}[\cite{shorack2009empirical}, p. 486]\label{lem:poisson deviations}
    Let $\zeta\sim \Poi(\lambda)$ be a Poisson random variable with mean $\lambda$. Then
\begin{align}\label{eq1:lem:poisson deviations}
\mathbb{P}(\zeta\le \lambda -x)&\le \exp\Big(-x -(\lambda-x)\log\Big(\frac{\lambda-x}{\lambda}\Big)\Big), \quad \forall  x \in (0,\lambda),\\
\label{eq2:lem:poisson deviations}
\mathbb{P}(\zeta\ge \lambda +x)&\le \exp\Big(x -(\lambda+x)\log\Big(\frac{\lambda+x}{\lambda}\Big)\Big), \quad \forall x>0.
\end{align}
\end{lemma}

\begin{lemma}[Poissonization]\label{lem:poissonization_trick}
Let $d \geq 2$, $\lambda>0$, and let $p = (p_1,\dots, p_d)$ be a probability distribution on $\Omega=\{1,\dots,d\}$.  Let $\widetilde n \sim \Poi(\lambda)$ and let $\widetilde Y=(\widetilde Y_1, \dots, \widetilde Y_d)$ be a random vector such that $\;\widetilde Y\,|\, \widetilde n \sim \mathcal{M}(p,\widetilde n)$. Then the entries $\widetilde Y_j$ are mutually independent and $\widetilde Y_j \sim \Poi(\lambda p_j)$ for all $j \in\{1,\dots, d\}$. 
\end{lemma}

\begin{proof}
For any $(n_1,\dots,n_d)\in \N^d$ we have
\begin{align*}
    \mathbb{P}(\widetilde Y_1 = n_1, \dots, \widetilde Y_d = n_d) &= \sum_{k=0}^\infty \mathbb{P}(\widetilde n = k)~\mathbb{P}\big( \widetilde Y_1 = n_1, \dots, \widetilde Y_d = n_d~|~\widetilde n = k\big) \nonumber
    \\
    & = \sum_{k=0}^\infty \Big(\frac{k!}{n_1! \cdots n_d!} \prod_{j=1}^d p_j^{n_j} \Big)\frac{e^{-\lambda}\lambda^k}{k!}\mathbbm{1}_{n_1+\cdots+n_d=k}
    \nonumber 
    \\
    &=\Big(\frac{\lambda^{n_1}\cdots \lambda^{n_d} }{n_1! \cdots n_d!} \prod_{j=1}^d p_j^{n_j} \Big) e^{-\lambda}  \sum_{k=0}^\infty \mathbbm{1}_{n_1+\cdots+n_d=k} 
    \nonumber
    \\
    & = \prod_{j=1}^d \frac{e^{-\lambda p_j} (\lambda p_j)^{n_j}}{n_j!}. 
\end{align*}
\end{proof}

\begin{proof}[Proof of Lemma \ref{lem_concent_bin}]
     The elements of the sum $Z_j = \sum_{i \in V_j} Y_{ij}$ are not mutually independent.  To overcome this difficulty, we use poissonization. 
    Let $\widetilde n \sim \Poi(n/2)$ and let $\widetilde Y | \widetilde n \sim \M(P, \widetilde n)$. 
    We define $\widetilde Z_j  = \sum_{i \in V_j} \widetilde Y_{ij}$ for any $j\in J$. 
    Lemma \ref{lem:poissonization_trick} implies that $\widetilde Y_{ij}\sim \Poi(nP_{ij}/2)$ for all $i,j$, and the random variables $(\widetilde Y_{ij})_{i,j}$ are mutually independent. It follows that $\widetilde Z_j$ has a Poisson distribution: $\widetilde Z_j \sim \Poi(\sum_{i \in V_j} nP_{ij}/2)$, and the random variables $(\widetilde Z_j)_j$ are mutually independent. At the same time, for any $k\ge 1$ the conditional distribution of $\widetilde Z_j$ given $\widetilde n=k$ coincides with the distribution of $Z_j$ under $Y \sim \M(P,k)$. 

    It follows from \eqref{eq1:lem:poisson deviations} with $x=\lambda/2$ and $\lambda = n\lambda_j/2$ that
$$
\mathbb{P}\left(\widetilde Z_j\le \frac{n \lambda_j}{4}\right) \leq \exp\Big(-\frac{n \lambda_j}{4}(1-\log 2)\Big)\le \exp\Big(-\frac{7}{2}(1-\log 2)\alpha\log (N)\Big)< N^{-\alpha}.
$$
Hence,
\begin{align*}
  \mathbb{P}\left(\forall j \in J: \widetilde Z_j\ge \frac{n \lambda_j}{4}\right) \geq 1- \frac{|J|}{N^{\alpha}}. 
\end{align*}
    On the other hand,
    \begin{align}
        \mathbb{P}\left(\forall j \in J: \widetilde Z_j\ge \frac{n \lambda_j}{4}\right) 
        & \leq \mathbb P\left(\widetilde n > n\right) + 
    \mathbb{P}\left(\forall j \in J: \widetilde Z_j\ge \frac{n \lambda_j}{4}, \; \text{and }\widetilde n \leq n\right) \nonumber
    \\
    & = \mathbb P\left(\widetilde n > n\right) + \sum_{k\le n}
    \mathbb{P}\left(\forall j \in J: \widetilde Z_j\ge \frac{n \lambda_j}{4} \;\Big|\; \widetilde n = k\right)\mathbb P\left(\widetilde n =k\right)\nonumber
    \\
    &= 
    \mathbb P\left(\widetilde n > n\right) + \sum_{k\le n}
    \mathbb{P}_{Z \sim \M(P,k)}\left(\forall j \in J: Z_j\ge \frac{n \lambda_j}{4} \right)\mathbb P\left(\widetilde n =k\right)\nonumber
    \\
    &\le  
    \mathbb P\left(\widetilde n > n\right) + 
    \mathbb{P}_{Y \sim \M(P,n)}\left(\forall j \in J: Z_j\ge \frac{n \lambda_j}{4} \right),
    \label{eq:lem:_concent_bin-0}
    \end{align}
    where we have used the inequality $\mathbb{P}_{Y \sim \M(P,k)}\left(\forall j \in J: Z_j\ge \frac{n \lambda_j}{4} \right)\le \mathbb{P}_{Y \sim \M(P,n)}\left(\forall j \in J: Z_j\ge \frac{n \lambda_j}{4} \right)$ that holds for all $k\le n$ due to stochastic dominance since, under $Y \sim \M(P,k)$, each $Y_{ij}$ has a binomial distribution with parameters $(P_{ij},k)$.
    It follows that
    \begin{align*}
        \mathbb{P}_{Y \sim \M(P,n)}\left(\forall j \in J: Z_j\ge \frac{n \lambda_j}{4}\right) 
        &\geq 1- \frac{|J|}{N^\alpha} - \mathbb P\left(\widetilde n > n\right).
    \end{align*}
    Applying \eqref{eq2:lem:poisson deviations} with $\zeta=\widetilde n$ and $\lambda=x=n/2$ we get
    \begin{align*}
        \mathbb P (\widetilde n > n) &\leq \exp(n(1/2-\log 2))\le \exp(14 (1/2-\log 2)\alpha \log(N))\leq \frac{1}{N^\alpha}. 
    \end{align*} 
    Combining the last two displays yields the lemma.
\end{proof}

The following lemma, that may be of independent interest,  provides a tool for deducing lower bounds in probability from lower bounds in expectation. 
\begin{lemma}[Deducing lower bound in probability from lower bound in expectation]\label{lem:reduction of lower bound}
   Let $\mathcal{P}_0$ be a metric space with metric $v:\mathcal{P}_0\times \mathcal{P}_0\to \mathbb{R}_+$, and let $\mathcal{P}$ be a subset of $\mathcal{P}_0$ with the property that there exists $U\in \mathcal{P}_0$ such that
   \begin{equation}\label{eq:lem:reduction of lower bound1}
      v(U,P){\leq } s, \quad \forall P\in  \mathcal{P},
   \end{equation}
   where $s>0$.
   Let $\{\mathbb{P}_P, P\in \mathcal{P} \}$ be a family of probability measures indexed by $\mathcal{P}$. Assume that
   \begin{equation}\label{eq:lem:reduction of lower bound2}
      \inf_{\widetilde P}\sup_{P \in \mathcal{P}}\E_P\, v(\widetilde P(Y),P) \ge as,
   \end{equation}
   where $a>0$, $\E_P$ denotes the expectation with respect to random variable $Y$ distributed according to $\mathbb{P}_P$, and
   $\inf_{\widetilde P}$ is the infimum over all estimators $\widetilde P$ that take values in $\mathcal{P}_0$. Then
   \begin{equation*}\label{eq:lem:reduction of lower bound3}
      \inf_{\widetilde P}\sup_{p \in \mathcal{P}}\mathbb{P}_P \big(v(\widetilde P(Y),P) \ge as/2\big) \ge a/6.
   \end{equation*}
\end{lemma}
\begin{proof}
 Consider the set $\mathcal{P}'= \{P' \in \mathcal{P}_0: \ v(P',P)\le 3s, \ \forall P\in \mathcal{P} \}$. This set is not empty since it contains $U$. Note that it is sufficient to consider estimators taking values in $\mathcal{P}'$, that is, for any estimator $\widetilde P$ there exists an estimator $\bar P$ with values in $\mathcal{P}'$ such that 
 \begin{equation}\label{eq:lem:reduction of lower bound4}
      v(\bar P,P)\le v(\widetilde P,P), \quad \forall P\in  \mathcal{P}.
   \end{equation}
   In fact, let $\widetilde P$ be any estimator. Define another estimator 
 \begin{align*}
    \bar P =  \begin{cases} 
    \widetilde P \ \  \text{if} \ \ v(\widetilde P,U) \le 2s,\\
    U \ \ \text{if} \ \ v(\widetilde P,U) > 2s.
    \end{cases} 
\end{align*}
Let us check that this estimator $\bar P$ satisfies \eqref{eq:lem:reduction of lower bound4} and $\bar P\in \mathcal{P}'$. Indeed,
if $v(\widetilde P,U) \le 2s$ then \eqref{eq:lem:reduction of lower bound4} obviously holds, and we have $v(\bar P,P)\le v(\widetilde P,U)+v(P,U)\le 3s$ for all $P\in  \mathcal{P}$. Otherwise, if $v(\widetilde P,U) > 2s$ then $v(\bar P,P)=v(U,P) \leq 2s-s < v(\widetilde P,U)-v(U,P)\le v(\widetilde P,P)$ for all $P\in  \mathcal{P}$.

It follows that 
\begin{align}\nonumber
  \inf_{\widetilde P}\sup_{p \in \mathcal{P}}\E_P\, v(\widetilde P,P)  &= \inf_{\widetilde P\in \mathcal{P}'}\sup_{p \in \mathcal{P}}\E_P\, v(\widetilde P,P),
  \\
  \label{eq:lem:reduction of lower bound5}
  \inf_{\widetilde P}\sup_{p \in \mathcal{P}}\mathbb{P}_P \big(v(\widetilde P,P) \ge as/2\big)&=\inf_{\widetilde P\in \mathcal{P}'}\sup_{p \in \mathcal{P}}\mathbb{P}_P \big(v(\widetilde P,P) \ge as/2\big).
\end{align}
Thus, we have
\begin{align*}
    as &\leq \inf_{\widetilde P\in \mathcal{P}'}\sup_{p \in \mathcal{P}} \E_P \,v(\widetilde P,P) \\
    & = \inf_{\widetilde P\in \mathcal{P}'}\sup_{p \in \mathcal{P}}\left\{~ \E_P\left[v(\widetilde P,P) \; \mathbbm{1}_{\left\{v(\widetilde P,P) \geq as/2\right\}}\right] ~+ ~ \E_P\left[v(\widetilde P,P)\; \mathbbm{1}_{\left\{v(\widetilde P,P) < as/2\right\}}\right]~\right\}\\
    & \leq 3s \inf_{\widetilde P\in \mathcal{P}'}\sup_{p \in \mathcal{P}}\mathbb{P}_P \big(v(\widetilde P,P) \ge as/2\big) + \, as/2,
\end{align*}
which together with \eqref{eq:lem:reduction of lower bound5}  implies the lemma.
\end{proof}

\begin{lemma}\label{prop:samp_comp_TV_discrete}
Let $Z_1,\dots,Z_m$ be iid random variables on a measurable space $(\mathcal Z,\mathcal U)$. Let $\widehat p_j=\frac{1}{m}\sum_{i=1}^m\mathbbm{1}_{Z_i \in A_j}$ and $p_j=\mathbb P (Z_1 \in A_j)$, where $A_j$, $j=1,\dots, \ell$, are disjoint subsets of $\mathcal X$.  
Then, for any $\delta\in (0,1)$,
\begin{align*}
    &\mathbb P\left(\sum_{j=1}^\ell |p_j - \widehat p_j| > \sqrt{\frac{\ell}{m}} + \sqrt{\frac{2\log(1/\delta)}{m}} \right) \le \delta.
\end{align*}
\end{lemma}
\begin{proof} Set $G(Z_1,\dots, Z_m):=\sum_{j=1}^\ell |p_j - \widehat p_j|$. We have 
 $$
 \mathbb E G(Z_1,\dots, Z_m) = \mathbb E \sum_{j=1}^\ell |p_j - \widehat p_j| \le \sum_{j=1}^\ell \left(\mathbb E [|p_j - \widehat p_j|^2]\right)^{1/2}
 \le \sum_{j=1}^\ell  \sqrt{\frac{p_j}{m}} \le \sqrt{\frac{\ell}{m}}.
 $$  
 Note that  $G(Z_1,\dots, Z_m)$ changes its value by at most $1/m$ if we replace any single $Z_i$ by another $Z_i'$. Therefore, by the bounded difference inequality (see, e.g.,~\cite[Theorem 2.2]{DevroyeLugosi}),
 \begin{align*}
    &\mathbb P\left(G(Z_1,\dots, Z_m) > \mathbb E G(Z_1,\dots, Z_m) + t\right)\le e^{-2t^2m}, \quad \forall t>0, 
\end{align*}
which yields the result.
\end{proof}
\begin{lemma}[Control of multinomial noise]\label{prop:mult_noise_2}
    Let $\alpha>1$, $N>1$, $n \in \N^*$,  let $P \in \R_+^{d_1\times d_2}$ be a matrix such that $\sum_{i,j} P_{ij} = 1$, and  let $Y \sim \M(P,n)$. 
   Consider an  extraction $Q$ of $P$ corresponding to two sets of indices $I \subseteq [d_1]$ and  $J \subseteq [d_2]$, that is $Q = (P_{ij})_{i\in I, j \in J}$.
    Let 
    $W_Q = (W_{ij})_{i\in I, j \in J}$, where $W = \frac{Y}{n} - P$ is the multinomial noise.
    Then 
    \begin{align*}
        \mathbb{P}\left(\|W_Q\|^2 \leq 9 \max\Big\{\frac{\alpha\normsquare{Q} \log N}{n},\, \Big(\frac{\alpha\log N}{n}\Big)^2\Big\}\right)\geq 1 - \frac{|I|+|J|}{N^\alpha}.
    \end{align*}
\end{lemma}

\begin{proof}
      We use matrix Bernstein inequality (see, for example, \cite[Exercise 5.4.15]{vershynin2018high}), which yields that 
for any independent zero mean matrices $M_1,\dots, M_n \in \R^{|I| \times |J|}$ such that, almost surely, $\|M_i\| \leq K, \forall i\in [n]$, we have 
\begin{align}\label{eq:matrixBernstein}
    \mathbb P\Bigg(\bigg\|\sum_{i=1}^n M_i\bigg\|> t\Bigg) \leq (|I|+|J|)\exp\left(-\frac{t^2/2}{\sigma^2  + Kt/3}\right), \quad \forall t>0,
\end{align}
where 
\begin{align*}
    \sigma^2 = \max\Bigg(\bigg\|\sum_{i=1}^n\mathbb E (M_iM_i^\top)\bigg\|, \bigg\|\sum_{i=1}^n \mathbb E (M_i^\top M_i)\bigg\|\Bigg).
\end{align*}
 We apply \eqref{eq:matrixBernstein} with $M_i = \big(X_i - P\big)_{I,J}$, where $X_i$'s are independent random matrices with distribution $\mathcal{M}(P,1)$. 
In this case, the inequality $\|M_i\| \leq K$ holds with $K=2$. To prove this, notice first that if $Q$ is an extraction of a probability matrix $P$, then $\|Q\|^2 \leq \normsquare{Q}$. 
Indeed, letting $p_i = \sum_{j\in J} P_{ij}, ~\forall i \in I$, and $p_{\max} = \max_{i\in I} p_i$, we obtain 
\begin{align*}
    \|Q\|^2 \leq \|Q\|_F^2 = \sum_{i\in I,j \in J} P_{ij}^2 \leq \sum_{i\in I} p_i^2 \leq p_{\max}\sum_{i\in I} p_i 
 \le p_{\max} \leq \normsquare{Q}.\label{eq:control_op_normsquare} 
\end{align*}
It follows that almost surely for all $i\in[n]$ we have
\begin{align*}
    \|M_i\| \leq \|(X_i)_{IJ}\| + \|Q\| \leq 1 + \sqrt{\normsquare{P}} \leq 2.
\end{align*}
Next,
\begin{align*}
    \left\|\sum_{i=1}^n\mathbb E( M_iM_i^\top)\right\| &= \left\|\sum_{i=1}^n\mathbb E [(X_i)_{IJ} (X_i)_{IJ}^\top] - Q Q^\top\right\| = \left\|n\operatorname{Diag}(p_k)_{k\in I} - QQ^\top\right\|\\
    & \leq n p_{\max}  \leq n\normsquare{Q},
\end{align*}
and controlling $\left\|\sum_{i=1}^n\mathbb E (M_i^\top M_i)\right\|$ analogously yields 
\begin{align*}
    \sigma^2 \leq n \normsquare{Q}.
\end{align*}
Now, we use the fact that $nW_Q$ has the same distribution as $\sum_{i=1}^n M_i$. Therefore, applying inequality \eqref{eq:matrixBernstein} and the above bounds on $\sigma^2$ and $K$ we obtain that, for all $t>0$,
\begin{align*}
     \mathbb P\left(\|nW_Q\|> t\right) &\leq (|I|+|J|)\exp\left(-\frac{t^2/2}{n \normsquare{Q}  + 2t/3}\right).
\end{align*}
Now, set
\begin{align*}
    t = 3\max\Big(\sqrt{\alpha n\normsquare{Q} \log(N)}, \ \alpha  \log(N)\Big).
\end{align*}
If $n\normsquare{Q} \geq \alpha \log(N)$ then $t=3\sqrt{\alpha n\normsquare{Q} \log(N)}\le 3n \normsquare{Q}$ and
\begin{align*}
     \mathbb P\left(\|nW_Q\|> t\right) 
     &\le (|I|+|J|)\exp\left(-\frac{t^2/2}{3 n \normsquare{Q}}\right)
      \leq (|I|+|J|)\exp\left(-\frac{3}{2}\alpha\log(N)\right)
      \leq (|I|+|J|) N^{-\alpha}.
\end{align*}
Otherwise, if $n\normsquare{Q} < \alpha \log(N)$ then $t=3\alpha\log(N)$ and the same bound holds:
\begin{align*}
     \mathbb P\left(\|nW_Q\|> t\right) 
     & \leq (|I|+|J|)\exp\left(-\frac{t^2/2}{3\alpha \log(N)}\right)
      \leq (|I|+|J|) N^{-\alpha}.
\end{align*}
\end{proof}

\vspace{3mm}

\textbf{Acknowledgment}: The work of Olga Klopp was funded by CY Initiative (grant ``Investissements d'Avenir'' ANR-16-IDEX-0008) and 
Labex MME-DII (ANR11-LBX-0023-01). The work of A.B. Tsybakov was supported by the grant of French National Research Agency (ANR) ``Investissements d'Avenir" LabEx Ecodec/ANR-11-LABX-0047. 
The authors are grateful to Laurène David, Gaëtan Brison, and Shreshtha Shaurya from Hi!PARIS for their help in implementing the algorithms of this paper.

\bibliographystyle{alpha}
\newcommand{\etalchar}[1]{$^{#1}$}


\end{document}